\documentclass[text=11pt]{article}

\usepackage{amsfonts,amsmath,amssymb,amsthm,commath,wasysym,mathrsfs}
\usepackage[utf8]{inputenc}
\usepackage[english]{babel}
\usepackage{bbm}
\usepackage{enumitem}
\usepackage{color}
\usepackage{comment}
\usepackage[margin=1in]{geometry}
\usepackage{setspace}
\usepackage{hyperref}
\usepackage{extarrows}
\usepackage{graphicx}
\usepackage[sort]{cite}

\newtheorem{theorem}{Theorem}

\newtheorem{proposition}[theorem]{Proposition}
\newtheorem{corollary}[theorem]{Corollary}
\newtheorem{definition}{Definition}
\newtheorem{lemma}[theorem]{Lemma}
\theoremstyle{remark}
\newtheorem{remark}{Remark}

\numberwithin{theorem}{section}
\DeclareMathOperator{\Z}{\mathbb{Z}}
\DeclareMathOperator{\R}{\mathbb{R}}

\DeclareMathOperator{\N}{\mathbb{N}}

\newcommand{\E}[1]{\mathbb{E}\left[#1\right]}
\newcommand{\eps}{\varepsilon}

\DeclareMathOperator{\pr}{\mathbb{P}}

\newcommand{\email}[1]{\href{mailto:#1}{#1}}

\title{Uniqueness of the infinite tree in low-dimensional random forests}

\author{Noah Halberstam and Tom Hutchcroft}

\begin{document}

\maketitle

\begin{abstract}
The \emph{arboreal gas} is the random (unrooted) spanning forest of a graph in which each forest is sampled with probability proportional to $\beta^{\# \text{edges}}$ for some $\beta\geq 0$, which arises as the $q\to 0$ limit of the Fortuin-Kastelyn random cluster model with $p=\beta q$.
	We study the infinite-volume limits of the arboreal gas on the hypercubic lattice $\Z^d$, and prove that when $d\leq 4$, any translation-invariant infinite volume Gibbs measure contains at most one infinite tree almost surely. Together with the existence theorem of Bauerschmidt, Crawford and Helmuth (2021), this establishes that for $d=3,4$ there exists a value of $\beta$ above which subsequential weak limits of the $\beta$-arboreal gas on tori have exactly one infinite tree almost surely. We also show that the infinite trees of any translation-invariant Gibbs measure on $\Z^d$ are one-ended almost surely in every dimension.
	The proof has two main ingredients:
	First, we prove a resampling property for translation-invariant arboreal gas Gibbs measures in every dimension, stating that the restriction of the arboreal gas to the trace of the union of its infinite trees is distributed as the  uniform spanning forest on this same trace. 
	Second, we prove that the uniform spanning forest of any translation-invariant random connected subgraph of $\Z^d$ is connected almost surely when $d\leq 4$.  This proof also provides strong heuristic evidence for the conjecture that the supercritical arboreal gas contains infinitely many infinite trees in dimensions $d\geq 5$. Along the way, we give the first systematic and axiomatic treatment of Gibbs measures for models of this form including the random cluster model and the uniform spanning tree.
\end{abstract}

\tableofcontents

\section{Introduction}
For each $\beta\geq 0$, the \textbf{$\beta$-arboreal gas} (a.k.a.\ the \emph{weighted uniform forest model}) on a finite undirected graph $G=(V,E)$ is a random subgraph $A$ of $G$ with probability mass function
\begin{equation} \label{eq:free_arboreal}
\pr_{\beta}(A=F)=\begin{cases}
	(1/Z_\beta) \beta^{\lvert F\rvert} &F\subseteq G \text{ is a spanning forest}\\
	0 &\text{otherwise}
\end{cases},\qquad\qquad Z_\beta = \sum_{F\subseteq G \text{ a spanning forest}} \beta^{\lvert F\rvert},
\end{equation}
 where $\abs{F}$ denotes the cardinality of the edge set of $F$ and a \emph{spanning forest} of $G$ is an acyclic subgraph of $G$ containing every vertex. Equivalently, the law of $A$ is equal to the law of Bernoulli percolation on $G$ with parameter $p=\beta/(1+\beta)$ conditioned to be acyclic.
 It is also equal to the $q\rightarrow 0$ limit of the $q$-state random cluster model with $p/q$ converging to $\beta$ \cite{pemantle2000towards,jacobsen2005spanning}, while its $\beta\rightarrow\infty$ limit is equal to the uniform spanning tree when $G$ is connected. (When $\beta=1$, the model is a uniform random spanning forest of $G$; this value of the parameter plays no special role in our analysis.) The arboreal gas is also closely related to various supersymmetric spin systems, which has led it to receive substantial attention in the physics literature \cite{caracciolo2004fermionic,caracciolo2007grassmann,caracciolo2017spanning,deng2007ferromagnetic}.
Despite these connections, there are very few tools available to study the model and several very basic conjectures about its behaviour have remained open for twenty years \cite{MR2060630}. See \cite{bauerschmidt2021spin,swan2021superprobability} for surveys of the model and its connections to other topics.

Interest in the arboreal gas has grown significantly in recent years following the breakthrough works of Bauerschmidt, Crawford, Helmuth and Swan \cite{MR4218682} and
 Bauerschmidt, Crawford and Helmuth \cite{PercTransition}, who studied the model's percolation phase transition through the lens of spontaneous symmetry breaking in an equivalent supersymmetric hyperbolic sigma model:
 In \cite{MR4218682} they proved that the arboreal gas on $\Z^2$ never contains any infinite trees for any finite $\beta<\infty$, while in \cite{PercTransition} they proved that the arboreal gas on $\Z^d$ contains infinite trees for sufficiently large values of $\beta$ when $d\geq 3$. (Stochastic domination by percolation easily implies that the arboreal gas does not contain infinite trees for \emph{small} values of $\beta$ in any dimension.) Since it remains open whether the arboreal gas is stochasically monotone in $\beta$ or in its boundary conditions, one must be careful to note some important subtleties in both statements: it is unclear whether there exist ``canonical'' definitions of the ``infinite-volume arboreal gas'' on $\Z^d$, and it is also unknown whether the existence of an infinite tree is monotone in $\beta$. A more precise statement of the results of \cite{MR4218682,PercTransition} is that any subsequential infinite-volume limit of the model on $\Z^2$ (with arbitrary boundary conditions) does not contain an infinite tree, while for $d\geq 3$ there exists $\beta_0=\beta_0(d)$ such that if $\beta>\beta_0(d)$ then any subsequential limit of the model on large $d$-dimensional tori contain at least one infinite tree almost surely. The authors also establish strong quantitative control of the model, showing in particular that the finite-cluster two-point function continues to display critical-like behaviour in the supercritical regime. (Similar phenomena have also been shown to occur for the arboreal gas on the complete graph \cite{luczak1992components,MR3845513} and on regular trees with wired boundary conditions \cite{ray2021forests,easo2022wired}, where the analysis of the critical-like behaviour of finite/non-giant clusters is more complete.)

The analysis of \cite{MR4218682,PercTransition} tells us nothing about the \emph{number} of infinite trees in the arboreal gas, 
 which is the main subject of this paper.
The analogous question has, however, been extensively studied for the \emph{uniform spanning tree}. Indeed, the seminal paper of Pemantle \cite{MR1127715} established that the uniform spanning tree of $\Z^d$ has a well-defined infinite-volume limit that is independent of the choice of boundary conditions and that is almost surely connected, i.e.\ a single tree, if and only if $d\leq 4$. This theorem was greatly generalized by Benjamini, Lyons, Peres, and Schramm \cite{MR1825141} who proved that the 
\emph{wired uniform spanning forest} (i.e.\ the infinite-volume limit of the uniform spanning tree with wired boundary conditions)
% wired uniform spanning forest
 of an infinite graph $G$ is connected almost surely if and only if two independent random walks on $G$ intersect infinitely often. This is known to occur for $G=\Z^d$ if and only if $d\leq 4$ by a classical theorem of Erd\"os and Taylor \cite{MR126299}. Since the uniform spanning tree is the $\beta \to \infty$ limit of the arboreal gas, it is natural to conjecture (see \cite[Page 8]{PercTransition}) that the same transition from uniqueness to non-uniqueness in four dimensions holds for the arboreal gas as in the uniform spanning tree.

In this paper we verify the low-dimensional case of this conjecture. Our proof also lends strong heuristic evidence to the high-dimensional case as we discuss later in the introduction.

 \begin{theorem}\label{thm:uniqueness}
 	For each $\beta>0$ and $d\in\{3,4\}$, every translation-invariant $\beta$-arboreal gas Gibbs measure on the Euclidean lattice $\Z^d$ is supported on configurations that have at most one infinite tree.
 \end{theorem} 

 Here, \emph{an arboreal gas Gibbs measure} on $\Z^d$ is any subsequential weak limit of arboreal gas measures on finite subgraphs of $\Z^d$ with (possibly random) boundary conditions; such Gibbs measures always exist by compactness, and \emph{translation-invariant} Gibbs measures always exist by taking  e.g.\ subsequential limits of the model with periodic boundary conditions.
  Let us stress that the structure of the set of Gibbs measures for the arboreal gas is very poorly understood, and, unlike the uniform spanning tree and ($q\geq 1$) random cluster model, it is not clear whether the free and wired infinite-volume measures are well-defined independently of the choice of exhaustion, or, for that matter, whether there is more than one Gibbs measure for the model at any value of $\beta$.  Indeed, an important contribution of our paper is to develop the first systematic, axiomatic treatment of Gibbs measures for models of this form (where the weight of a configuration depends on its connectivity properties), as discussed in more detail below.

  \begin{remark}
The proof of Theorem~\ref{thm:uniqueness} also applies in dimensions $d\leq 2$, but the result is vacuous in this case since the model has no infinite clusters for any $\beta<\infty$ by the results of \cite{MR4218682}. (While the main theorem of that paper is written only for subsequential limits of the model with free boundary conditions, the proof applies  with arbitrary boundary conditions).
  \end{remark}
 
 Theorem \ref{thm:uniqueness}   has
  the following corollary in conjunction with the aforementioned results of \cite{PercTransition} (translation-invariance being an automatic feature of subsequential limits of automorphism-invariant models on tori).
 \begin{corollary}\label{thm:ex_plus_uniq}
 	Fix a dimension $d\in\{3.4\}$ and $\beta>0$, and for each $n\geq 1$ let $\pr_n$ be a $\beta$-arboreal gas measure on the $d$-dimensional torus of side length $n$. There exists a constant $\beta_0=\beta_0(d)>0$ such that if $\beta>\beta_0$ then every subsequential weak limit of the sequence $(\pr_n)_{n\geq 1}$ is supported on configurations that contain a unique infinite tree. 
 \end{corollary} 

\begin{remark}
Theorem~\ref{thm:uniqueness} also implies an analogue of Corollary~\ref{thm:ex_plus_uniq} for (subsequential) double limits of the model on the torus with an external field as considered in \cite{PercTransition}, where one first sends the size of the torus to infinity and then takes the external field to zero. This is because any such subsequential limit is a translation-invariant Gibbs measure for the model, as follows from a straightforward modification of the proof of Proposition~\ref{prop:alternative_char}.
\end{remark}

 % \medskip
 
\noindent \textbf{About the proof.} We now briefly overview the proof of Theorem~\ref{thm:uniqueness}.
 Unlike \cite{MR4218682,PercTransition}, which exploit a non-probabilistic equivalence between the arboreal gas and a supersymmetric sigma model, our methods are purely probabilistic.  
 Our argument can be divided into two parts, which we now describe in turn. Both parts of the proof lead to intermediate results of independent interest.

\medskip

\noindent \textbf{Augmented Gibbs measures and the resampling property.} The first part of the paper, which is valid in any dimension, establishes a relationship between the infinite trees in the arboreal gas and the wired uniform spanning forest of a certain random subgraph of $\Z^d$. This part of the paper is mostly ergodic-theoretic in nature, and works by studying the properties of the space of translation-invariant Gibbs measures.

 \begin{theorem} \label{thm:resampling}
 	Let $d\geq 1$ and $\beta>0$ and let $A$ be distributed as a translation-invariant $\beta$-arboreal gas Gibbs measure on $\Z^d$. If we define $I_\infty$ to be the set of vertices that belong to the infinite components of $A$ and define $\mathrm{Tr}(I_\infty)$ to be the subgraph of $\Z^d$ induced by $I_\infty$ then the following hold:
    \begin{enumerate}
        \item $\mathrm{Tr}(I_\infty)$ is connected almost surely.
        \item The conditional distribution of the restriction of $A$ to $\mathrm{Tr}(I_\infty)$ given $I_\infty$ and the restriction of $A$ to $\mathrm{Tr}(I_\infty^c)$ is almost surely equal to the law of the wired uniform spanning forest of $\mathrm{Tr}(I_\infty)$.
    \end{enumerate}
 \end{theorem}

The second part of this theorem can be rephrased equivalently in terms of resampling: If we first sample the arboreal gas $A$ then take $F'$ to be a random variable sampled according to the law of the wired uniform spanning forest on $\mathrm{Tr}(I_\infty)$, then the forest formed from $A$ by deleting all the infinite trees of $A$ and adding in the trees of $F'$ has the same distribution as $A$ itself.

In the process of proving this theorem we develop a new axiomatic framework for infinite-volume Gibbs measures of the arboreal gas, with the usual DLR theory of Gibbs measures \textit{not} being applicable to the arboreal gas due to a failure of `quasilocality' of the Hamiltonian.
% \red{specification measurability}
 Our replacement for this theory, which is developed in Section~\ref{section:arboreal_gas}, 
  revolves around what we term \textit{augmented subgraphs}. Roughly speaking, this means that we enrich our random variables so that they include information about which vertices are connected to each other -- possibly ``through infinity'' -- outside of each finite set.  We remark that previous papers on related models including the random cluster model and the uniform spanning tree have sidestepped the development of such a framework (in part because they tend to be focused on the free and wired measures, which we do not know are well-defined for the arboreal gas), and we are optimistic that the framework we develop will also be useful in the future study of those models. See Remarks~\ref{remark:Gibbs} and \ref{remark:UST2} for further discussion.

 It will already be clear to experts that the first part of Theorem~\ref{thm:resampling} is a kind of Burton-Keane \cite{burton1989density} theorem for the induced subgraph $\operatorname{Tr}(I_\infty)$.
  More interestingly, the second part of the theorem also hides a second Burton-Keane argument `under the hood': To prove it, we first show that a similar resampling theorem holds where one replaces $I_\infty$ by the infinite classes of the \emph{augmented} connectivity relation (so that, \emph{a priori}, one must sample the wired uniform spanning forest separately on the trace of each such class), before employing an ``augmented'' Burton-Keane argument to prove that there is in fact only one infinite augmented connectivity class almost surely.

This argument clearly demonstrates the utility of our perspective on the arboreal gas in terms of augmented subgraphs and augmented Gibbs measures.
A further demonstration is given by the following theorem on the almost-sure one-endedness of infinite trees in the arboreal gas,
   which drops out neatly once the surrounding framework has been established.  Here, an infinite tree is said to be \textbf{one-ended} if there is exactly one infinite simple path starting at each vertex. The same theorem has also been established for the uniform spanning tree via very different methods~\cite{MR1127715,MR1825141}.
\begin{theorem}\label{thm:ends}
Let $d\geq 1$ and $\beta>0$ and let $A$ be distributed as a translation-invariant $\beta$-arboreal gas Gibbs measure on $\mathbb{Z}^d$. Then every infinite tree in $A$ is one-ended almost surely.
\end{theorem}

\begin{remark}
Theorem~\ref{thm:resampling} allows us to import `for free' various ergodic-theoretic theorems from the uniform spanning tree  to the arboreal gas. For example, the indistinguishability theorem of \cite{hutchcroft2017indistinguishability} can be immediately applied to get that the infinite trees of the arboreal gas are indistinguishable when they exist, and a similar statement holds for the ``multicomponent indistinguishability theorem'' of \cite{hutchcroft2020indistinguishability}. This may be useful for studying more refined properties of the arboreal gas in high dimensions, as the multicomponent indistinguishability theorem plays an important role in the study of the adjacency structure of trees in the high-dimensional uniform spanning forest \cite{hutchcroft2019component,benjamini2011geometry}.
\end{remark}

\medskip

\noindent \textbf{Connectivity of the UST in low-dimensional unimodular random graphs.} Theorem~\ref{thm:resampling} reduces the study of the infinite trees in the arboreal gas to the study of the uniform spanning forest of the induced subgraph $\mathrm{Tr}(I_\infty)$, which is a translation-invariant random subgraph of $\Z^d$. When $d\geq3$ and $\beta$ is very large we have by the results of \cite{PercTransition} that $I_\infty$ has density very close to $1$ (at least for subsequential limits of the arboreal gas on tori), so that it is reasonable to think of $\mathrm{Tr}(I_\infty)$ as a ``small perturbation'' of the original hypercubic lattice $\Z^d$. It seems very unlikely that this small perturbation would lead to any drastic difference in the behaviour of the random walk, which supports the conjecture that the number of infinite trees in the trees in the arboreal gas and uniform spanning tree should be the same, at least for $\beta$ very large. Unfortunately it \emph{is} possible in general for a high-density translation-invariant random induced subgraph of $\Z^d$ to have very different large-scale random walk behaviour than that of the full lattice, so that to implement this argument rigorously in the high-dimensional case one must use features of the arboreal gas beyond its translation invariance. The problem is made particularly delicate by the slow decay of correlations in the model \cite{PercTransition}, which make it difficult to compare $\mathrm{Tr}(I_\infty)$ to a better-understood model such as Bernoulli site percolation.

While we have not yet been able to circumvent this problem in the high-dimensional case, the low-dimensional case is more tractable since, informally, ``the monotonicity goes in the right direction'': we think of the connectivity of the wired uniform spanning forest (which, as previously mentioned, is equivalent to two independent random walks intersecting infinitely often almost surely) as a ``small graph'' property, so that it is plausibly preserved when taking ``reasonable'' subgraphs.
Unfortunately, despite this intuition, it is still not literally true that every connected subgraph of $\Z^d$ has a connected wired uniform spanning forest when $d\leq 4$. Indeed, the subgraph of $\Z^3$ induced by the union of the origin with the two half-spaces $\{(x,y,z):x>0\}$ and $\{(x,y,z):x<0\}$ has two components in its wired uniform spanning forest almost surely. Moreover,
 it follows from a theorem of Thomassen \cite[Theorem 3.3]{10.1214/aop/1176989708} that $\Z^d$ contains a transient tree for every $d\geq 3$, and it is easily seen that the wired uniform spanning forest of any such tree has infinitely many components almost surely.

The second part of the paper, which is specific to the low-dimensional case, establishes that, in contrast to these examples, the wired uniform spanning forest is always connected almost surely in any \emph{translation-invariant} random subgraph of $\Z^d$ when $d\leq 4$.  We state a simple special case of the relevant theorem now, with a significant generalization given in Theorem~\ref{thm:spanning_tree_general}.

 \begin{theorem} \label{thm:spanning_tree}
Let $d\leq 4$, let $S$ be a translation-invariant random subset of $\Z^d$ and let $\mathrm{Tr}(S)$ be the subgraph of $\Z^d$ induced by $S$. Then the wired uniform spanning forest of each infinite connected component of $\mathrm{Tr}(S)$ is connected almost surely.
 \end{theorem}

 The proof of this theorem draws mostly on random walk techniques, and is inspired in particular by previous work on \emph{collisions} of random walks in unimodular random graphs \cite{MR3399814,MR4364738}.

\begin{remark}
Translation-invariant random subgraphs of $\Z^d$ do \emph{not} always have disconnected wired uniform spanning trees when $d\geq 5$, even when these graphs are induced by connected sets of vertices. (Indeed, starting with a random space-filling curve one can construct such a translation-invariant random induced subgraph that is a.s.\ rough-isometric to $\Z$.) This suggests that a more delicate approach is required to understand the number of infinite trees in the high-dimensional arboreal gas.
\end{remark}

 \begin{remark}
We believe that the theory we develop in this paper can be applied with minor modifications to prove analogous uniqueness theorems for a number of similar random forest models in dimensions $d\leq 4$. For example, it should apply to the variant of the arboreal gas in which the forest is required to contain at most one non-singleton component, which is a kind of `dilute spanning tree' model.\footnote{This model always has infinite-volume limits containing infinite trees when $\beta>1$, even when $d=1$. Indeed, in this regime the contribution to the partition function from a single spanning tree is larger than that from all configurations with a sublinear number of edges, so that most the contribution to the partition function comes from configurations with a linear number of edges. 
The actual critical value should be smaller than $1$.
 This is related to the results of \cite{dereudre2022fully}.} Indeed, this model should actually be significantly simpler to study via our methods than the arboreal gas, since (in the language of Section~\ref{section:arboreal_gas}) its Gibbs augmentations trivially have at most one non-singleton augmented connectivity class almost surely. The main (easily addressed) complication is that the definition of an augmented Gibbs measure needs to be modified so that the random variables are also enriched with the data of which finite subgraphs have a non-singleton component in their complement, and which boundary vertices (if any) belong to this component. 
  We do not pursue such generalizations further in this paper.
 \end{remark}
 
\begin{remark}
All our methods generalize immediately to arbitrary transitive graphs of at most four-dimensional volume growth. The resampling theorem, Theorem~\ref{thm:resampling}, can be extended much more generally to every \emph{amenable} transitive graph. One noteworthy consequence of this is as follows: In \cite{MR4218682}, Bauerschmidt, Crawford, Helmuth, and Swan prove that the arboreal gas on $\Z^2$ cannot have a \emph{unique} infinite tree for any $\beta<\infty$ (since the probability that $x$ is connected to $y$ is small when $x-y$ is large), then deduce that there are no infinite trees almost surely using a Burton-Keane argument \emph{on the model's planar dual}. The first part of their argument does not use planarity, and also applies to quasi-transitive graphs such as slabs which are quasi-isometric to $\R^2$ but not planar. An appropriate generalization of our Theorem~\ref{thm:uniqueness} can be used to replace the second part of their argument, so that the entire result holds without planarity.
\end{remark}

 \section{Gibbs measures and augmented subgraphs} \label{section:arboreal_gas}

 In this paper, we are primarily concerned with weak limits of finite-volume arboreal gas measures on infinite graphs $G$. In order to proceed, it is desirable to have an axiomatic characterization of these infinite-volume measures, which will make it easier to apply ergodic-theoretic arguments. Unfortunately, the usual DLR--Gibbs theory (as described in e.g.\ \cite{friedli2017statistical,le2008introduction}) is not applicable to these measures: given a limit measure $\mu$, a random variable $A\sim \mu$ and a finite box $H\subset G$, the law of the restriction of $A$ to $H$ conditioned on $A\cap H^c$ cannot, \emph{a priori}, be expressed as a function of $A\cap H^c$. This is because when we take the limit, connectivity information is lost and we do not know which infinite trees in $A$ should be regarded as connected ``through infinity'' to which other infinite trees.
 
 In this section, we develop an \textit{augmented Gibbs framework} which rectifies this problem. A central idea is to make the appropriate long-range connectivity information available locally by enriching the space that our random variables are defined in.
  In the next section, we use this framework to prove the resampling property for translation-invariant Gibbs measures, Theorem \ref{thm:resampling}. 

\begin{remark}
\label{remark:Gibbs}
  As mentioned earlier, we believe that the theory of augmented Gibbs measures we develop here should be useful to the study of 
   other probabilistic statistical physics models such as the uniform spanning tree and random cluster model, which are also incompatible with the standard DLR framework for the same reasons as in our setting. Indeed, is is notable that no abstract theory of Gibbs measures has previously been developed for these models despite their broad popularity. For example, 
   in Glazman and Manolescu's work on the structure of the set of Gibbs measures for the random cluster model on $\Z^2$ 
    \cite{glazman2021structure}, the authors consider only an (\emph{a priori}) special class of Gibbs measures in which infinite clusters are always considered to be connected at infinity. As discussed in \cite[Remark 1.5]{glazman2021structure}, considering only this restricted class of Gibbs measures has various downsides, including that this class is not (\emph{a priori}) preserved under planar duality. Our definition of Gibbs measures for models of this form is given strong justification by the fact that it coincides with the set of all possible limits of the models in finite-volume, with arbitrary boundary conditions, and is more general than that of \cite{glazman2021structure}.
     The two notions can be shown to coincide for the random cluster model in the translation-invariant case, but it is currently unclear whether the two notions will coincide without the assumption of translation-invariance. For the uniform spanning tree, a version of the Gibbs property was proposed by Sheffield \cite{sheffield2006uniqueness}, which has the non-standard property that it describes the conditional distribution of the restriction of the tree to a finite set given both what is outside the set \emph{and how the points on the boundary of the set are connected inside the set}; our definition is more standard in that it describes the distribution of what is inside the set given information only about what is outside. Further discussion of how our theory applies to the UST appears in Remarks~\ref{remark:UST1} and \ref{remark:UST2}.
 \end{remark}

 \subsection{Definitions}
 We begin by setting up some necessary notation which will be used throughout the rest of the paper before defining augmented subgraphs and arboreal gas Gibbs measures.
 
 \medskip

 \textbf{Graph notation.}
 For any graph $G=(V,E)=(V[G],E[G])$, and vertices $u,v\in V[G]$, we write $u\sim_G v$ if $\{u,v\}\in E[G]$, write $u\xleftrightarrow{G}v$ if the vertices $u$ and $v$ are in the same connected component of $G$, and write $G(v)$ for the connected component of $G$ containing $v$. 
 For any graph $G$, write $\mathcal{S}(G)$ for the set of subgraphs of $G$ (which we take to be pairs of subsets of $V$ and $E$) and write $\mathcal{S}^f(G)$ for the set of finite subgraphs of $G$. We will always assume that all graphs $G$ are \textit{locally finite}, meaning that all their vertex degrees are finite.  For any graph $G$, an increasing sequence of finite subgraphs of $G$ whose union is the entire graph is called an \textbf{exhaustion} of $G$.

\medskip

 \textbf{Finite-volume arboreal gas Gibbs measures.}
Let $G=(V,E)$ be a  countable, locally finite graph $G=(V,E)$ and let $H\subset G$ be a finite subgraph of $G$. We define the \textbf{inner vertex boundary} $\partial H$
 to be the set of vertices of $H$ that are incident to an edge of $G$ that does not belong to $H$. (If $H$ is an \emph{induced} subgraph of $G$ then $\partial H$ is equal to the set of vertices of $H$ that are adjacent to a vertex of $V[G]\setminus V[H]$.)
  For each set $S$ we write $\mathcal{P}[S]$ for the set of equivalence relations on $S$, which we encode as functions $\phi:S\times S\rightarrow\{0,1\}$ such that $\phi(x,y)=1$ if and only if $x$ and $y$ are in the same equivalence class. For each $\phi\in\mathcal{P}(\partial H)$ and subgraph $H^\prime\subseteq H$, we write $H^\prime/\phi$ for the graph constructed by taking $H^\prime$ and identifying the sets of vertices in $V[H^\prime]\cap\partial H$ which belong to the same equivalence class of $\phi$, deleting any self-loops created by this identification. These equivalence relations will serve as boundary conditions, keeping track of connectivity outside of $H$.   We write $\mathcal{F}(H)$ for the set of spanning forests of $H$, i.e.\ the set of acyclic subgraphs of $H$ containing every vertex of $H$ and, given an equivalence relation $\phi\in \mathcal{P}[\partial H]$, we say a forest $F\in \mathcal{F}(H)$ \textbf{extends} $\phi$ if $F/\phi$ is acyclic. We write $\mathcal{F}(H,\phi)=\mathcal{F}(G,H,\phi)$ for the subset of forest subgraphs of $H$ which extend $\phi$ and say that such a forest is an \textbf{$(H,\phi)$-maximal spanning forest} if it contains every vertex of $H$ and there is no edge in $E[H]$ which can be added to $F$ to yield another element of $\mathcal{F}(H,\phi)$. We write $\mathcal{F}_T(H,\phi)$ for the set of $(H,\phi)$-maximal spanning forests; when $H/\phi$ is connected, maximal spanning forests of $H/\phi$ are the same thing as spanning trees of $H/\phi$.
 
 % \medskip

 For each $\beta\in[0,\infty)$, we define the \textbf{finite-volume $\beta$-arboreal gas Gibbs measure} on a finite subgraph $H$ of $G$ with \textbf{boundary condition} $\phi\in\mathcal{P}(\partial H)$ by
 \[
 \pr_{H,\beta}^\phi(F)= \pr_{G,H,\beta}^\phi(F)=\begin{cases}
   (1/Z_\beta^\phi) \beta^{\lvert F\rvert} &F\in\mathcal{F}(H,\phi)\\
   0 &\text{otherwise}
 \end{cases},\qquad\qquad Z_\beta^\phi = \sum_{F\in\mathcal{F}(H,\phi)} \beta^{\lvert F\rvert}.
 \]
 (In particular, when $\beta=0$ this measure puts all its mass on the subgraph of $H$ with no edges.)
 We remark that if every equivalence class of $\phi$ contains just a single element then this measure coincides with the free arboreal gas measure on $H$. We also define the
  % \textbf{uniform spanning tree}
  \textbf{finite-volume $\infty$-arboreal gas Gibbs measure} on $H$ with boundary condition $\phi$ by
 \[
 \pr_{H,\infty}^\phi(F)= \pr_{G,H,\infty}^\phi(F)=\begin{cases}
   \abs{\mathcal{F}_T(H,\phi)}^{-1}  &F\in\mathcal{F}_T(H,\phi)\\
   0 &\text{otherwise},
 \end{cases}
 \]
 which is the weak limit of $\pr_{H,\beta}^\phi$ as $\beta\rightarrow\infty$ and can be identified with the uniform measure on maximal spanning forests of $H/\phi$. In particular, when $H/\phi$ is connected, this measure can be identified with the uniform spanning tree measure on $H/\phi$.
 More generally, given $\beta \in [0,\infty]$, a finite subgraph $H\in \mathcal{S}^f(G)$, and a probability measure $\nu$ on $\mathcal{P}(\partial H)$, we write $\pr_{H,\beta}^\nu$ for the measure with probability mass function
\[
 \pr_{H,\beta}^\nu(F) = \sum_{\varphi \in \mathcal{P}(\partial H)} \nu(\varphi) \pr_{H,\beta}^\phi(F),
\]
which we call a \textbf{finite-volume $\beta$-arboreal gas Gibbs measure} with boundary condition $\nu$. Probabilistically, this measure is the law of the configuration obtained by first sampling a random boundary condition according to the (arbitrary) distribution $\nu$, then sampling the arboreal gas with this boundary condition. Considering random boundary conditions in this way has the advantage that it automatically makes all the sets of measures we consider convex. 

% \medskip

The finite-volume version of the Gibbs property for these measures is as follows: Given a finite subgraph $H$ and a probability measure $\nu$ on the set of equivalence relations on $\partial H$, let $\phi$ be a random variable with law $\nu$ and, given $\phi$, let $A$ be a random variable with conditional law $\mathbb{P}_{H,\beta}^\phi$, so that $A$ has marginal law $\mathbb{P}_{H,\beta}^\nu$. If $H'$ is a subgraph of $H$ and we define an equivalence relation $\Phi(H')$ on $\partial H'$ by taking $u$ and $v$ to be in the same class of $\Phi(H')$ if they are connected in $(A\setminus E[H'])/\phi$, then
\begin{equation}
\label{eq:finite_Gibbs}
\mathbb{P}^\nu_{H,\beta}( A \cap H' = \cdot \mid A \setminus E[H'], \phi ) = \mathbb{P}^{\Phi(H)}_{H',\beta}( A = \cdot\,).
\end{equation}
In words, the conditional law of $A \cap H'$ given $A\setminus E[H']$ and $\phi$ is equal to $\mathbb{P}^{\Phi(H)}_{H',\beta}$.
This identity is an immediate consequence of the definitions, and encapsulates the intuition that what happens outside of $H'$ affects the distribution of $A$ inside $H'$ only in so far as it determines which boundary vertices of $H'$ are connected to each other \emph{outside} of $H'$. Note that \eqref{eq:finite_Gibbs}  is exactly the same Gibbs property enjoyed by the random cluster model; most of the theory we develop in the rest of this section will also apply straightforwardly to any other model satisfying this same form of the Gibbs property in finite volume.

 \medskip

 We now move on to defining the space of \emph{augmented subgraphs}, which allow us to meaningfully extend the Gibbs property \eqref{eq:finite_Gibbs} to infinite-volume measures. To avoid trivialities, we take care to make sure all relevant definitions continue to work as expected in the case that $G$ is finite or disconnected.

\medskip

\noindent
 \textbf{The space of augmented subgraphs.}
Let $G=(V,E)$ be a locally finite graph. We define an \textbf{augmented subgraph} of $G$ to be a pair $(S,\Phi)$ where $S$ is a subgraph of $G$ and $\Phi$ is a collection $(\Phi(H):H \in \mathcal{S}^f(G))$, where $\Phi(H)$ is an equivalence relation on $\partial H$ for each $H \in \mathcal{S}^f(G)$, satisfying the consistency condition
 \begin{equation}\tag{Con}
 \begin{array}{c}
 \text{For every $H,K\in\mathcal{S}^f(G)$ with $H\subset K$ and $u,v\in \partial H$, } \\
 \text{$u$ and $v$ are  related in $\Phi(H)$ if and only if they are connected in $(S \cap K\setminus E[H])/\Phi(K)$,}\end{array}
 \end{equation}
 where vertices that do not belong to a subgraph are considered to not be connected to any other vertex in that subgraph.
 % \[
% u 
 % \]
We interpret $\Phi(H)$ as dictating connectivity \emph{outside} of $H$: the consistency condition states that if two vertices in the boundary of $H \subseteq K$ are connected outside of $H$ according to $\Phi(H)$, then these two vertices must also be connected outside of $H$ according to $S\cap K$ and $\Phi(K)$, and vice versa.
 Given an augmented subgraph $(S,\Phi)$ of $G$,  
  we define the \textbf{augmented connectivity relation} by $u\xleftrightarrow{(S,\Phi)}v:=\Phi(\{u,v\})(u,v)$, where here $\{u,v\}$ is the graph consisting of the vertices $u,v$ and no edges, so that, by consistency, 
  \[u \xleftrightarrow{(S,\Phi)} v \text{ if and only if $u$ is connected to $v$ in $(H \cap S)/\Phi(H)$}\] 
  for each (and hence every) finite subgraph $H\in\mathcal{S}^f(G)$ containing both $u$ and $v$.

  We write $\mathcal{A}(G)$ for the space of augmented subgraphs of $G$, which we endow with its natural product topology and associated Borel sigma-algebra, so that $\mathcal{A}(G)$ is compact and the projection map $\pi:\mathcal{A}(G)\to \mathcal{S}(G)$ defined by $\pi:(S,\Phi)\mapsto S$ is continuous. 
 We call an augmented subgraph $(S,\Phi)$ with underlying subgraph $S$ an \textbf{augmentation} of $S$, and call $\Phi$ the \textbf{boundary map} of the augmentation $(S,\Phi)$. Every subgraph $S$ of $G$ admits boundary maps $\Phi_\mathrm{free}=\Phi_\mathrm{free}^S$ and $\Phi_\mathrm{wired}=\Phi_\mathrm{wired}^S$ defined by
  \begin{align}
\Phi_\mathrm{free}(H)(u,v)&=1 \iff \hspace{1.5cm}\text{$u$ and $v$ are connected in $S \setminus E[H]$}\label{eq:free_augmentation}\\
% \intertext{and}
% \vspace{-1cm}
\text{and}\qquad \Phi_\mathrm{wired}(H)(u,v)&=1 \iff \begin{array}{c}  \text{$u$ and $v$ are connected in $S \setminus E[H]$ or both} \\\text{ belong to infinite connected components of $S\setminus E[H]$},\end{array} \label{eq:wired_augmentation}
  \end{align}
which are distinct whenever $S$ has more than one infinite connected component or more than one end. We call the resulting augmentations $(A,\Phi_\mathrm{free})$ and $(A,\Phi_\mathrm{wired})$ the \textbf{free} and \textbf{wired} augmentations of $A$. (We warn the reader that the relationship between these augmentations and the usual terminology for free and wired Gibbs measures for the uniform spanning tree is not as straightforward as one might hope; see Remark~\ref{remark:UST2}.) These augmentations are extremal in the sense that the equivalence classes of an arbitrary augmentation contain those of the free augmentation and are contained in those of the wired augmentation. In general a subgraph may admit a very large number of distinct augmentations.

\medskip
 
\noindent \textbf{Augmentations are determined by their tails.} We now discuss a key property of augmented subgraphs that will be used throughout our analysis. Let $(S,\Phi)$ be an augmented subgraph of a locally finite graph $G$. The consistency property implies that if $H$ and $H'$ are two finite subgraphs of $G$ with $H\subseteq H'$, then $\Phi(H)$ is determined by $\Phi(H')$ and $S$. In particular, if for each finite subgraph $H$ of $G$ we define
\[
\Phi_H = (\Phi(K) : K \text{ is a finite subgraph of $G$ containing $H$}), 
\]
then the full augmented subgraph $(S,\Phi)$ is completely determined by the pair $(S,\Phi_H)$ for each finite subgraph $H$ of $G$. \emph{This gives us a well-defined notion of what it means to add or delete finitely many edges from an augmented subgraph $(S,\Phi)$:} Given an augmented subgraph $(S,\Phi)$ and two disjoint finite sets of edges $A$ and $B$, we define an augmented subgraph $(S,\Phi) \cup A \setminus B$ by taking $H$ to be a finite subgraph of $G$ containing both $A$ and $B$ and extending $(S \cup A \setminus B, \Phi_H)$ to a full augmented subgraph by consistency; it is easily verified that this definition does not depend on the choice of finite subgraph $H$.

\medskip

\noindent \textbf{Infinite-volume arboreal gas Gibbs measures.} We now define infinite-volume Gibbs measures for the arboreal gas. (NB: Although we emphasize the infinite-volume case, the definition also works in finite volume.)
 Given a random augmented subgraph $(A,\Phi)$ of a countable, locally finite graph $G$ and a finite subgraph $H$ of $G$, we write $\mathcal{G}_H$ for the sigma-algebra generated by $A \setminus E[H]$ and $\Phi_H$, which represents the data of the augmented subgraph that is determined `outside of $H$'.
 
 \begin{definition}\label{definition:IVGM}
   Let $G$ be a countable, locally finite graph and fix $\beta\in [0,\infty]$. We say that a probability measure $\pr_\beta$ on $\mathcal{F}(G)$ is a \textbf{$\beta$-arboreal gas Gibbs measure of $G$} if there exists 
    % (with respect to the product topology)
     a probability measure $\mathbb{Q}_\beta$ on $\mathcal{A}(G)$, such that the following hold:
   \begin{enumerate}
      \item The pushforward $\pi_* \mathbb{Q}_\beta$ is equal to $\pr_\beta$. In other words, if $(A,\Phi)\sim \mathbb{Q}_
      \beta$ then $A\sim \mathbb{P}_\beta$.
      \item  If $(A,\Phi)$ is a random variable distributed as $\mathbb{Q}_\beta$ and $H$ is a finite subgraph of $G$, then the conditional law of $A\cap H$ given $\mathcal{G}_H$ is almost surely equal to $\pr_{H,\beta}^{\Phi(H)}$. 
   \end{enumerate}
   We will refer to the second property as the \textbf{augmented Gibbs property}. We call any measure $\mathbb{Q}_\beta$ which satisfies these two properties a \textbf{Gibbs augmentation} of $\mathbb{P}_\beta$, and call any measure $\mathbb{Q}_\beta$ on $\mathcal{A}(G)$ satisfying the second of these two properties an \textbf{augmented $\beta$-arboreal gas Gibbs measure}.
   
 \end{definition}

 We will often refer to $\infty$-arboreal gas Gibbs measures as \textbf{uniform spanning tree Gibbs measures} or \textbf{uniform maximal spanning forest Gibbs measures} (the former terminology not always being appropriate when $G$ is not connected).

 \medskip

This axiomatic definition has the advantage that it is well-suited to ergodic-theoretic techniques. That it is an \emph{appropriate} definition is justified by the following alternative characterisation of infinite-volume arboreal gas measures, as presented in the introduction. 

 \begin{proposition}\label{prop:alternative_char}
 Let $G$ be an infinite, countable, locally finite graph.
   For each $\beta\in[0,\infty]$, the $\beta$-arboreal gas Gibbs measures of $G$ are exactly the subsequential weak limits of finite-volume $\beta$-arboreal gas Gibbs measures -- with possibly random boundary conditions -- on exhaustions of $G$.
 \end{proposition}

 We note that for any $\beta\in[0,\infty]$, any exhaustion $(H_n)_{n\geq 0}$ of $G$ and any sequence of probability measures on boundary conditions $(\nu_n)_{n\geq 1}$, the sequence of measures $(\pr_{H_n,\beta}^{\nu_n})_{n\geq 1}$ will always have at least one subsequential weak limit by compactness of $\mathcal{A}(G)$.

 \begin{proof}[Proof of Proposition~\ref{prop:alternative_char}]
   Fix $\beta\in [0,\infty]$. We first check that  any $\beta$-arboreal gas Gibbs measure $\pr_\beta$ is a subsequential weak limit of finite-volume $\beta$-arboreal gas Gibbs measures with possibly random boundary conditions. Let $(A,\Phi)$ be a random variable with the law of a Gibbs augmentation of $\pr_\beta$ and let $(H_n)_{n\geq 1}$ be any exhaustion of $G$. By the Gibbs property, the law of $A$ restricted to $H_n$ is equal to the law of $\pr_{H_n,\beta}^{\,\nu_n}$, where $\nu_n$ is the law of $\Phi(H_n)$, and so the weak limit of the sequence $(\pr_{H_n,\beta}^{\,\nu_n})_{n\geq 1}$ of finite volume $\beta$-arboreal gas Gibbs measures with random boundary conditions is equal to $\pr_\beta$.
   
   We now show the converse.
   Let $(H_n)_{n\geq 1}$ be an exhaustion of $G$, let $(\nu_n)_{n\geq 1}$ be a sequence of probability measures on equivalence relations on $\partial H_n$, and suppose that the sequence $(\mathbb{P}_{H_n,\beta}^{\nu_n})$ converges to some limit measure $\mathbb{P}_\beta$. For each $n\geq 1$ let $\phi_n$ be an equivalence relation on $\partial H_n$ with law $\nu_n$, let $A_n$ be a random variable with conditional law $\pr_{H_n,\beta}^{\phi_n}$ (so that $A_n$ has marginal law $\pr_{H_n,\beta}^{\nu_n}$), and for each finite subgraph $H$ of $G$ define an equivalence relation $\Phi_n(H)$ on $\partial H$ by setting
\[
\Phi_n(H)(u,v)= \begin{cases}  \mathbbm{1}(\text{$u$ and $v$ are connected in $A_n/\phi_n(H)$}) & H \subseteq H_n\\
1 & \text{otherwise.}
\end{cases}
\]
   By compactness, taking a subsequence if necessary, $(A_n,\Phi_n)$ converges weakly to some random variable $(A,\Phi)$, where $A$ has law $
   \mathbb{P}_\beta$. Using \eqref{eq:finite_Gibbs}, one can check from the definitions that $\Phi$ is almost surely an augmentation of $A$ and that the law of $(A,\Phi)$ is a Gibbs augmentation of $\mathbb{P}_\beta$, completing the proof. 
 \end{proof}

\noindent\textbf{The uniform spanning tree.} 
Let $G$ be an infinite, connected, locally finite graph. For each finite subgraph $H$ of $G$, we define the \textbf{free boundary condition} $\mathrm{f}=\mathrm{f}_H\in\mathcal{P}(\partial H)$ to be the equivalence relation whose classes all have cardinality one and define \textbf{wired boundary condition} $\mathrm{w}=\mathrm{w}_H$ on $H$ to be the equivalence relation on $\partial H$ in which all points are related. It was proven implicitly by Pemantle \cite{MR1127715} that if $(H_n)_{n\geq 1}$ is any exhaustion of $G$ by finite subgraphs then the two sequences $(\mathbb{P}_{H_n,\infty}^{\mathrm{f}})_{n\geq 1}$ and $(\mathbb{P}_{H_n,\infty}^{\mathrm{w}})_{n\geq 1}$ have well-defined weak limits that do not depend on the choice of exhaustion $(H_n)_{n\geq 1}$; these limits are known as the \textbf{free} and \textbf{wired uniform spanning forest} measures on $G$. It follows from the $\beta=\infty$ case of Proposition~\ref{prop:alternative_char} that if $G$ is a connected, locally finite graph then the free and wired uniform spanning forests on $G$ are indeed Gibbs measures for the uniform spanning tree on $G$. Moreover, these two measures are always stochastically maximal and minimal among the set of all Gibbs measures for the uniform spanning tree on $G$ as made precise in the following lemma.

 \begin{lemma}\label{lemma:infty_free_wired}
    Let $G$ be a connected, locally finite graph and let $\pr$ be a Gibbs measure for the uniform spanning tree on $G$. 
    Then $\pr$ is stochastically dominated by the free uniform spanning forest on $G$ and stochastically dominates the wired uniform spanning forest on $G$. In particular, if the free and wired uniform spanning forest of $G$ coincide then $G$ has a unique Gibbs measure for the uniform spanning tree.
 \end{lemma}

 \begin{proof}
    Let $(V_n)_{n\geq 1}$ be an increasing sequence of subsets of $V[G]$ converging to $V[G]$, and for each $n\geq 1$, let $H_n=\mathrm{Tr}[V_n]$ be the subgraph of $G$ induced by $V_n$.
It follows from the negative associated theorem of Feder and Mihail \cite{feder1992balanced} (see also \cite[Theorem 4.6 and Exercise 10.8]{LP:book}) that the measure $\pr^{\phi}_{H_n,\infty}$ is stochastically decreasing in $\phi$ in the sense that if $\phi_1,\phi_2$ are two equivalence relations with $\phi_1$ a refinement of $\phi_2$ then $\pr^{\phi_1}_{H_n,\infty}$ stochastically dominates $\pr^{\phi_2}_{H_n,\infty}$. It follows in particular that every measure of the form $\pr^\nu_{H_n,\infty}$ is stochastically dominated by $\mathbb{P}_{H_n,\infty}^{\mathrm{f}}$ and stochastically dominates $\mathbb{P}_{H_n,\infty}^{\mathrm{w}}$. The claim follows by taking limits in light of this and Proposition~\ref{prop:alternative_char}.
 \end{proof}

\begin{remark}\label{remark:UST1}
Pemantle \cite{MR1127715} established implicitly that the free and wired uniform spanning forests of $\Z^d$ coincide for every $d\geq 1$.
  In general, a graph $G$ has a unique Gibbs measure for the uniform spanning tree if and only if it does not admit any non-constant harmonic functions of finite Dirichlet energy \cite{MR1825141}, which holds in particular for every amenable transitive graph \cite[Corollary 10.9]{LP:book} as well as in many nonamenable examples. See \cite[Chapter 10]{LP:book} for detailed background.
 \end{remark}

 \begin{remark}\label{remark:UST2}
Naively, one might like to say that the augmentation we need to put on the free uniform spanning forest to make its law into an augmented Gibbs measure is precisely the free augmentation as defined in \eqref{eq:free_augmentation}, while the augmentation we need to wired uniform spanning forest to make its law into an augmented Gibbs measure is precisely the wired augmentation as defined in \eqref{eq:wired_augmentation}. This intuition is correct when $G$ is, say, a $3$-regular tree, but is false in general. Indeed, consider the hypercubic lattice $\Z^d$, where the free and wired uniform spanning forest measures coincide for every dimension $d\geq 1$ as discussed above. In one dimension (where the spanning tree is just the entire line), the correct augmentation to place on the infinite-volume uniform spanning tree is the free augmentation; using the wired augmentation does not work, since under this augmentation the conditional probability that any edge is present given that all other edges are present would be zero, not one. In dimensions two to four the infinite-volume limit is supported on configurations with a single one-ended tree, and there is no choice in how to define the augmentation. In dimension five and higher, where there are infinitely many one-ended trees, the correct augmentation to use is the wired augmentation; using the free augmentation does not work since the Gibbs property would imply that an edge connecting two distinct infinite trees must be present with probability~$1$. (In other examples, such as the free uniform spanning forest on the free product $\Z^5 * \Z_2$, neither the free nor the wired augmentations are appropriate.) As a historical note, let us remark that this subtlety in how to correctly define the Gibbs property for uniform spanning forests led to an error in the work of Burton and Pemantle \cite{burton1993local} which was not discovered until a decade later by Lyons \cite{lyons2005asymptotic} and corrected in the work of Sheffield \cite{sheffield2006uniqueness}.
 \end{remark}

 \subsection{Translation-invariant Gibbs measures}

 In this section we refine our focus to translation-invariant Gibbs measures on $\Z^d$. In particular, we will discuss how each such Gibbs measure can be decomposed in terms of \emph{extremal} translation-invariant Gibbs measures, which have better ergodicity properties.
 In the usual DLR--Gibbs formalism for (quasi)local systems such as the Ising model, it is a standard result that any Gibbs measure can be decomposed as a mixture of \textit{tail-trivial} Gibbs measures, which assign probability $0$ or $1$ to any event in the tail-sigma algebra. Indeed, in this framework, the tail-trivial Gibbs measures are exactly the extremal points of the convex set of Gibbs measures and so the desired decomposition is an immediate corollary of Choquet's theorem. An analogous result also holds for translation-invariant Gibbs measures (see Remark \ref{remark:TT}), which can always be decomposed into a mixture of \textit{ergodic} translation-invariant Gibbs measures; these are the measures that assign probability $0$ or $1$ to all translation-invariant events.
While the first of these results translates directly to our setting, we were not able to prove the direct analogue of the second result, and instead prove a slightly weaker result that will suffice for our later applications.

\medskip

\textbf{Tail triviality.} We begin by discussing \emph{tail triviality}, where the relevant theory holds for arbitrary graphs.
Let $G$ be a countable, locally finite graph, and recall that for each finite subgraph $H$ of $G$ we define $\mathcal{G}_H$ to be the sigma-algebra of Borel sets $E$ in $\mathcal{A}(G)$ such that an augmented subgraph $(S,\Phi)$'s belonging to $E$ is determined by $S\setminus H$ and $\Phi_H:=(\Phi(H'):H'$ a finite subgraph of $G$ containing $H)$. We define the \textbf{tail sigma-algebra} $\mathcal{T}$ on $\mathcal{A}(G)$ to be the intersection $\bigcap_H\mathcal{G}_H$ taken over all finite subgraphs $H$ of $G$. 

\begin{lemma}\label{lem:condition_on_tail}
Let $G=(V,E)$ be a countable, locally finite graph, let $\beta\in [0,\infty]$, and let $\mathbb{Q}_\beta$ be an augmented $\beta$-arboreal gas Gibbs measure on $G$. If $X \in \mathcal{T}$ is a tail event with $\mathbb{Q}_\beta(X)>0$, then the conditional measure $\mathbb{Q}_\beta(\,\cdot\,|X)$ is an augmented $\beta$-arboreal gas Gibbs measure on $G$.
\end{lemma}

\begin{proof}[Proof of Lemma~\ref{lem:condition_on_tail}]
Let $\mathbb{Q}_X:=\mathbb{Q}_\beta(\,\cdot\,|X)$. Since $X$ is $\mathcal{G}_H$ measurable for each finite subgraph $H$ of $\Z^d$, we have for each such subgraph and each subgraph $F$ of $H$ that
    % and observe that for each subgraph $F$ of $H$ we have that
   \begin{align*}
   \hspace{1cm}\mathbb{Q}_X(A \cap H = F \mid \mathcal{G}_H) &= \mathbb{Q}_\beta(A \cap H = F \mid \mathcal{G}_H)  &&\hspace{-1cm}\text{a.s.}
   % \]
   \intertext{and hence by the augmented Gibbs property of $\mathbb{Q}_\beta$ that}
   % \[
   \hspace{1cm}\mathbb{Q}_X(A \cap H = F \mid \mathcal{G}_H) &= \mathbb{P}_{H,\beta}^{\Phi(H)}(A \cap H = F)  &&\hspace{-1cm}\text{a.s.}
   \end{align*}
   for every finite subgraph $H$ of $\Z^d$ and every subgraph $F$ of $H$, which is precisely the augmented Gibbs property for $\mathbb{Q}_X$.
\end{proof}

\begin{corollary}
\label{cor:tail_trivial}
Let $G=(V,E)$ be a countable, locally finite graph and let $\beta\in [0,\infty]$. Every extremal element of the convex set of augmented $\beta$-arboreal gas Gibbs measures on $G$ is tail-trivial in the sense that it gives every tail event probability $0$ or $1$.
\end{corollary}

\begin{proof}[Proof of Corollary~\ref{cor:tail_trivial}]
If $\mathbb{Q}_\beta$ is a $\beta$-arboreal gas Gibbs measure and $X\in\mathcal{T}$ is such that $\mathbb{Q}(X)\in(0,1)$ then, by Lemma~\ref{lem:condition_on_tail}, we can write $\mathbb{Q}_\beta$ as a convex combination of $\beta$-arboreal gas Gibbs measures
   $\mathbb{Q}(\cdot) = \mathbb{Q}(\cdot|X)\mathbb{Q}(X)+ \mathbb{Q}(\cdot|X^c)\mathbb{Q}(X^c)$.
   Clearly $\mathbb{Q}(\cdot|X)$ and $\mathbb{Q}(\cdot|X^c)$ are non-identical as they each assign a different probability to $X$, so that $\mathbb{Q}_\beta$ is not extremal.
\end{proof}

Let $\mathcal{M}_\beta=\mathcal{M}_\beta(G)$ denote the set of all augmented $\beta$-arboreal gas Gibbs measures on $G$.
 Since $\mathcal{M}_\beta$ is a compact convex subspace of the space of all signed measures on $\mathcal{A}(\Z^d)$, which is a locally-convex topological vector space with respect to the weak (a.k.a.\ weak*) topology, we may apply Choquet's theorem \cite{simon2011convexity} to get that for each $\mathbb{Q}_\beta \in \mathcal{M}_\beta$ there exists a measure $\nu$ on the set of extremal points $\mathrm{ext}(\mathcal{M}_\beta)$ such that
   \[
   \mathbb{Q}_\beta(\cdot) = \int_{\mathrm{ext}(\mathcal{M}_\beta)} \mathbb{Q}^\prime_\beta(\cdot)\  d\nu (\mathbb{Q}^\prime_\beta).
   \]
 Probabilistically, this means that every augmented $\beta$-arboreal gas Gibbs measure can be sampled by first sampling a random \emph{tail trivial} augmented $\beta$-arboreal gas Gibbs measure of appropriate distribution, then sampling from this random tail-trivial measure. Unfortunately this result has limited applicability to our setting since we are interested primarily in the translation-invariant case, and it is not guaranteed that a translation-invariant augmented Gibbs measure decomposes as a mixture of \emph{translation-invariant} tail-trivial augmented Gibbs measures.

 \begin{remark}\label{remark:TT}
One can use the Krein-Milman theorem \cite{simon2011convexity} to prove that every extremal $\beta$-arboreal gas Gibbs measure can be expressed as a weak limit over finite-volume Gibbs measures with \emph{non-random} boundary conditions. We omit the details of these arguments since we are interested primarily in the translation-invariant setting.
 \end{remark}

% \medskip

\textbf{Translation invariance and ergodicity.}
 We now fix a dimension $d\geq 2$ and, as usual, abuse notation by writing $\Z^d$ both for the set of $d$-tuples of integers and the hypercubic lattice considered as a graph, writing $E_d$ for the associated set of nearest-neighbour edges in $\Z^d$. For each $x\in\Z^d$, we define the translation operator $\tau_x$ on subgraphs of $\Z^d$ as \[\tau_x\big((V,E)\big)=(\{v+x:v\in V\},\{\{v_1+x,v_2+x\}:\{v_1,v_2\}\in E\}).\]
 For each $x\in\Z^d$, $\tau_x$ also acts on augmented subgraphs via
 $\tau_{-x}(S,\Phi)=(\tau_{-x}S,\tau_{-x}\Phi)$ where $[\tau_{-x}\Phi](H)(u,v)=\Phi(H+x)(u+x,v+x)$. \textbf{Translation-invariant events} in, and \textbf{translation-invariant measures} on $\mathcal{S}(G)$ and $\mathcal{A}(G)$ are then defined as expected with respect to these operations. We write $\mathcal{I}$ for the sigma-algebra of translation-invariant events in $\mathcal{A}$ and write $\mathcal{I}_S$ for the sigma-algebra of translation-invariant events in $\mathcal{A}$ depending only on the subgraph coordinate (that is, for which any two augmentations of the same subgraph either both belong to the event or both belong to its complement).

The following lemma implies that if we wish to study translation-invariant Gibbs measures, it suffices to consider translation-invariant \emph{augmented} Gibbs measures. 

 \begin{lemma}\label{lemma:m_to_aug}
   Fix $d\geq 1$, $\beta\in(0,\infty]$, and let $\pr_\beta$ be a $\beta$-arboreal gas infinite-volume Gibbs measure on $\Z^d$. Then $\pr_\beta$ is translation-invariant if and only if it admits a translation-invariant Gibbs augmentation.
 \end{lemma}

 \begin{proof} The `if' direction is trivial; we focus on the `only if' direction, which follows from the amenability of $\Z^d$.
 Let $\mathbb{P}$ be a translation-invariant infinite volume Gibbs measure and let $(A,\Phi)$ have the law of an augmentation of $\pr$. For each $n\geq 1$, let $V_n$ be a uniformly chosen vector in $\Lambda(n)$, and consider the sequence of random variables $(\tau_{V_n}A,\tau_{V_n}\Phi)_{n\geq 1}$.
 Taking a subsequential weak limit yields a translation-invariant random variable $(A^\prime,\Phi^\prime)$ whose law is a Gibbs augmentation of $\pr$. (Alternatively, one can check that for each $\beta$-arboreal gas Gibbs measure $\mathbb{P}_\beta$ on $\Z^d$, the set of Gibbs augmentations of $\mathbb{P}_\beta$ is a weakly compact convex subset of the space of probability measures on augmented subgraphs of $\Z^d$. When $\mathbb{P}_\beta$ is translation-invariant this set is fixed by the action of $\mathbb{Z}^d$, and therefore must contain a fixed point since $\mathbb{Z}^d$ is amenable.)
 \end{proof}

We write $\mathcal{M}^T_\beta=\mathcal{M}^T_\beta(\Z^d)$ for the set of translation-invariant $\beta$-arboreal gas Gibbs measures on $\Z^d$, which is a weakly closed, convex set of the space of all signed measures on $\mathcal{A}(\Z^d)$. Applying Choquet's theorem as above yields that every element of $\mathcal{M}_\beta^T$ can be written as a mixture of its extremal points: For each $\mathbb{Q}_\beta \in \mathcal{M}^T_\beta$ there exists a measure $\nu$ on the set of extremal points $\mathrm{ext}(\mathcal{M}^T_\beta)$ such that
   \[
   \mathbb{Q}_\beta(\cdot) = \int_{\mathrm{ext}(\mathcal{M}^T_\beta)} \mathbb{Q}^\prime_\beta(\cdot)\  d\nu (\mathbb{Q}^\prime_\beta).
   \]
In the standard quasilocal DLR--Gibbs theory, one would then argue that every element of $\mathrm{ext}(\mathcal{M}^T_\beta)$ is ergodic, meaning that it assigns probability $0$ or $1$ to every invariant event in $\mathcal{A}$. Unfortunately, the standard proof of this fact breaks down in our setting. More specifically, it is not clear whether the translation-invariant sigma-algebra is always contained in the completion of the tail sigma-algebra. Nevertheless, we do still have that extremal translation-invariant Gibbs measures are trivial on the intersection of the tail and invariant sigma algebras:

\begin{lemma}\label{lem:condition_on_invariant}
Fix $d\geq 1$ and $\beta\in[0,\infty]$. If $\mathbb{Q}_\beta\in \mathcal{M}^T_\beta$ is a translation-invariant augmented $\beta$-arboreal gas Gibbs measure and $X \subseteq \mathcal{A}$ is an event belonging to the $\mathbb{Q}_\beta$-completions of both $\mathcal{T}$ and $\mathcal{I}$ with $\mathbb{Q}_\beta(X)>0$ then $\mathbb{Q}_\beta(\cdot|X)$ is also a translation-invariant $\beta$-arboreal gas Gibbs measure.
\end{lemma}

\begin{proof}
Since $X$ is in the completion of $\mathcal{T}$, there exists an event $X'\in \mathcal{T}$ with $\mathbb{Q}_\beta(X\Delta X')=0$ and hence with $\mathbb{Q}_\beta(\cdot|X)=\mathbb{Q}_\beta(\cdot|X')$, so that Lemma~\ref{lem:condition_on_tail} implies that $\mathbb{Q}_\beta(\cdot|X)$ is an augmented $\beta$-arboreal gas Gibbs measure. Similarly, since $X$ is in the completion of  $\mathcal{I}$, there exists an event $X''\in \mathcal{I}$ such that $\mathbb{Q}_\beta(\cdot|X)=\mathbb{Q}_\beta(\cdot|X'')$, and one may verify from the definitions that $\mathbb{Q}_\beta(\cdot|X'')$ is translation-invariant since both $\mathbb{Q}_\beta$ and $X''$ are.
\end{proof}

\begin{corollary}\label{cor:extremal_translation_invariant}
Fix $d\geq 1$ and $\beta\in[0,\infty]$. If $\mathbb{Q}_\beta\in \operatorname{ext}(\mathcal{M}^T_\beta)$ is an extremal translation-invariant augmented $\beta$-arboreal gas Gibbs measure and $X \subseteq \mathcal{A}$ is an event belonging to the $\mathbb{Q}_\beta$-completions of both $\mathcal{T}$ and $\mathcal{I}$ then $\mathbb{Q}_\beta(X)\in \{0,1\}$.
\end{corollary}

This corollary together with the next lemma implies that the sigma-algebra $\mathcal{I}_S$ of translation-invariant events that are insensitive to the choice of augmentation is always trivial for any extremal translation-invariant augmented Gibbs measure. This is a (slightly unsatisfactory) analogue of the statement in the standard DLR--Gibbs theory that extremal translation invariant measures are ergodic.

 \begin{lemma} \label{lemma:translation_tail}
   Fix $d\geq 1$, $\beta\in(0,\infty]$, and let $\mathbb{Q}_\beta$ be a translation-invariant augmented $\beta$-arboreal gas Gibbs measure on $\Z^d$. Then $\mathcal{I}_S$ is contained in the $\mathbb{Q}_\beta$-completion of $\mathcal{T}$. That is, for any translation-invariant $X\in \mathcal{I}_S$, there exists 
   $Y\in \mathcal{T}$ such that $\mathbb{Q}_\beta(X\Delta Y)=0$.
\end{lemma}

\begin{proof}[Proof of Lemma~\ref{lemma:translation_tail}]
   Let $(A,\Phi)$ be distributed as $\mathbb{Q}_\beta$ and for each $n\geq 1$ let $\Lambda_n$ be the box $[-n,n]^d$ considered as a subgraph of $\Z^d$. By definition of the product Borel sigma-algebra, $\sigma(A)$ is generated by
   the union $\bigcup_H \sigma(A \cap H)$, where this union is taken over all finite subgraphs $H$ of $\Z^d$.
   Since $\mathcal{I}_S = \mathcal{I}\cap \sigma(A) \subseteq \sigma(A)$, 
    it follows from the Dynkin $\pi-\lambda$ theorem that for every event $X\in\mathcal{I}_S$ and every $\eps>0$ there exists a finite subgraph $H$ of $\Z^d$ and an event $X' \in \sigma(A \cap H)$ such that $\mathbb{Q}_\beta(X'\Delta X)\leq \eps$. Fix an event $X\in \mathcal{I}_S$ and for each $n\geq 1$ let $H_n \in \mathcal{S}^f(\Z^d)$ and $X_n\in \sigma(A\cap H_n)$ be such that $\mathbb{Q}_\beta(X\Delta X_n)\leq 2^{-n}$.
     For each $n\geq 1$, let $X_n^\prime=\tau_{x_n}(X_n)$, where $x_n \in \mathbb{Z}^d$ is such that $\tau_{x_n}(H_n)$ is disjoint from $\Lambda_n$.
   We observe that $\mathbb{Q}(X\Delta X_n^\prime)=\mathbb{Q}(X\Delta X_n)\leq 2^{-n}$ by translation-invariance of $X$ and $\mathbb{Q}_\beta$, and moreover that $X_n'\in\sigma(A \setminus \Lambda_n)\subseteq\mathcal{G}_{\Lambda_n}$
   for every $n\geq 1$. Letting $X''=\limsup X'_n:=\cap_{n\geq 1}\cup_{m\geq n} X_m^\prime$ be the event that infinitely many of the events $X'_n$ hold, we have that $X''\in\mathcal{T}$ and that
   \[
   \mathbb{Q}(X\cap X'')\leq\mathbb{Q}\Big(X\Delta X_n^\prime \text{ holds for infinitely many $n$} \Big)\leq \lim_{n\rightarrow\infty} \sum_{m\geq n} 2^{-m}=0,
   \]
   which completes the proof.
\end{proof}

\begin{remark}
This proof does \textit{not} straightforwardly extend to show that $\mathcal{I}$ is contained in the completion of $\mathcal{T}$ due to the  long-range dependencies encoded in the boundary map. It would be possible to run the proof if one knew that $\sigma(A)$ and $\mathcal{T}$ together generate the entire sigma algebra on $\mathcal{A}(G)$, but this seems to be a surprisingly subtle matter.
\end{remark}

We deduce the following immediate corollary.

\begin{corollary}
\label{cor:A_ergodic}
Fix $d\geq 1$ and $\beta\in[0,\infty]$. If $\mathbb{Q}_\beta\in \operatorname{ext}(\mathcal{M}^T_\beta)$ is an extremal translation-invariant augmented $\beta$-arboreal gas Gibbs measure then $\pi_* \mathbb{Q}_\beta$ is an ergodic translation-invariant $\beta$-arboreal gas Gibbs measure.
\end{corollary}

\begin{remark}
 We will later prove in Corollary~\ref{cor:augmentation_always_wired} that if $(A,\Phi)$ is distributed as an a translation-invariant augmented $\beta$-arboreal gas Gibbs measure on $\Z^d$ with $\beta<\infty$, the boundary map $\Phi$ is almost surely equal to the wired boundary map associated to $A$, and hence coincides a.s.\ with a measurable function of $A$. Moreover, the boundary map also coincides a.s.\ with a measurable function of $A$ in the case $\beta=\infty$ as discussed in Remark~\ref{remark:UST2}.
  As such, it follows \emph{a posteriori} (see Corollary~\ref{cor:full_ergodicity}) that the completions of the sigma-algebras $\mathcal{I}$ and $\mathcal{I}_S$ are equal, and hence that every measure in $\mathrm{ext}(\mathcal{M}^T_\beta)$ is ergodic. Let us stress however that this proof uses specific properties of the arboreal gas (and, implicitly, the amenability of $\Z^d$), in contrast to the other proofs of this section which apply without change to a very large class of models with connection-based interactions. Moreover, the logical structure of the paper means that we cannot assume true ergodicity in the proof of Theorem~\ref{thm:resampling} since this ergodicity is established only at the very end of Section~\ref{section:resampling}.
  \end{remark}

\begin{remark}
It follows by standard arguments that the extremal elements of the set of \emph{all} translation-invariant measures on $\mathcal{A}(\Z^d)$ are ergodic, and hence by Choquet theory that every translation-invariant measure on $\mathcal{A}(\Z^d)$ can be written as a mixture of ergodic translation-invariant measures. This statement is of limited use to us since we prefer to stay within the class of augmented arboreal gas Gibbs measures.
\end{remark}

\section{Proof of Theorems~\ref{thm:resampling} and \ref{thm:ends}}
\label{section:resampling}

In this section we use the framework developed in the previous section to prove Theorems \ref{thm:resampling} and \ref{thm:ends}.
We begin with Theorem~\ref{thm:resampling}, whose proof is split into two propositions. The first, proven in Section~\ref{subsec:resampling_general}, establishes a `local' version of the same resampling theorem that does not require the symmetry of $\Z^d$, while the second, proven in Section~\ref{subsec:infinite_locus}, establishes the basic qualitative features of the augmented connectivity relation for augmented arboreal gas Gibbs measures on $\Z^d$. As a part of the proof of Section~\ref{subsec:infinite_locus} we prove Theorem \ref{thm:ends}, which states that all the infinite trees in the arboreal gas are one-ended almost surely.

 \subsection{Resampling without symmetry}
\label{subsec:resampling_general}

In this section we prove the following proposition, which establishes a very general version of the resampling property that does not require any symmetry assumptions on the graph or the measure. This proposition is inspired in part by the UST resampling theorem of Lyons, Peres, and Sun \cite{lyons2020induced}.

\begin{proposition} \label{prop:local_resampling_general}
Let $G=(V,E)$ be a connected, locally finite graph, let $o$ be a vertex of $G$, and let $(A,\Phi)$ be distributed as an augmented $\beta$-arboreal gas Gibbs measure on $G$. Let $I_o=\{x \in V:o\xleftrightarrow{(A,\Phi)} x\}$ and let $\mathrm{Tr}(I_o)$ be the subgraph of $G$ induced by $I_o$.
 Then the conditional distribution of the restriction of $A$ to $I_o$ given $I_o$ and the restriction of $A$ to the complement of $I_o$ is almost surely equal to some Gibbs measure for the uniform maximal spanning forest on $\mathrm{Tr}(I_o)$, where the choice of Gibbs measure may be random.
\end{proposition}

 \begin{proof}[Proof of Proposition~\ref{prop:local_resampling_general}]
    We begin by observing that a related resampling property holds in finite volume. Let $H$ be a finite subgraph of $G$, so that $\Phi(H)$ is an equivalence relation on $\partial H$.
    For each forest $F\in\mathcal{F}(H,\Phi(H))$, let $T_o[F]$ be the connected component of $o$ in $F$ considered as a subgraph of $H/\Phi(H)$, let $I_o[F]=I_o^{H,\Phi(H)}[F]$ be the vertex set of $T_o[F]$, and let $\mathrm{Tr}(I_o[F])$ be the subgraph of $H/\Phi(H)$ induced by $I_o[F]$.
    We make three observations. First, note that $T_o[F]$ is always a spanning tree of $\mathrm{Tr}(I_o[F])$. Second, note that if we let $T^\prime$ be any other spanning tree of $\mathrm{Tr}(I_o[F])$ and let $F^\prime$ be formed from $F$ by deleting $T_o[F]$ and adding $T'$, then $I_o[F^\prime]=I_o[F]$. Finally, we observe that the probability $\pr_{H,\beta}^{\phi_n}$ assigns to forests $F\in\mathcal{F}(H,\phi_n)$ depends only on the cardinality of their edge sets, so that
    $\pr_{H,\beta}^{\phi_n}(F)=\pr_{H,\beta}^{\phi_n}(F^\prime)$. Putting these observations together gives that if $F\sim\pr_{\Lambda_n,\beta}^{\Phi(H)}$, then conditional on $I_o[F]$ and the restriction of $F$ to the complement of $I_o[F]$, the restriction of $F$ to $I_o[F]$ is distributed as the uniform spanning tree on $\mathrm{Tr}(I_o[F])/\Phi(H)$. 

    By the augmented Gibbs property, it follows that the conditional distribution of the restriction of $A$ to $I_o[A\cap H]=I_o^{H,\Phi(H)}[F]$ given $\mathcal{G}_H$, $I_o[A\cap H]$, and the restriction of $A$ to the complement of $I_o[A \cap H]$ is almost surely equal to the uniform spanning tree measure on $\mathrm{Tr}(I_o[A\cap H])/\Phi(H)$. In particular, this conditional distribution depends only on $\Phi(H)$ and $I_o[A\cap H]$. Moreover, the consistency property of the boundary map $\Phi$ implies that $I_o[A\cap H]=I_o^{H,\Phi(H)}[A \cap H]$ is equal to the intersection of $I_o$ with the vertex set of $H$.
     Thus, if for each finite subgraph $H$ of $G$ we define $\mathcal{F}_H$ to be the sigma-algebra generated by $\mathcal{G}_H$, $I_o \cap V[H] = I_o^{H,\Phi(H)}[A\cap H]$, and the restriction of $A$ to the complement of $I_o$, then the conditional law of the restriction of $A$ to $I_o \cap V[H]$ given $\mathcal{F}_H$ is a.s.\ equal to the uniform spanning tree measure on $\mathrm{Tr}(I_o \cap H)/\Phi(H)$. Since this law depends only on $I_o\cap H$ and $\Phi(H)$, it follows that the conditional law of the restriction of $A$ to $I_o \cap V[H]$ given $I_o$ and the restriction of $A$ to the complement of $I_o$ is almost surely of the form $\mathbb{P}_{\mathrm{Tr}(I_o \cap V[H]),\infty}^\nu$ for some probability measure $\nu$ on the boundary of $\mathrm{Tr}(I_o \cap V[H])$ in $\mathrm{Tr}(I_o)$, where the measure $\nu$ is determined by the conditional distribution of $\Phi(H)$ given this information. Taking a limit as $H$ exhausts $G$ and using Proposition~\ref{prop:alternative_char} yields the claim. \qedhere

 \end{proof}

 \subsection{The structure of the augmented connectivity relation}
\label{subsec:infinite_locus}

In this section we prove the following proposition about the structure of the augmented connectivity relation in a translation-invariant arboreal gas Gibbs measure on $\Z^d$ and then deduce Theorem~\ref{thm:resampling} from this proposition together with Proposition~\ref{prop:local_resampling_general}.

\begin{proposition}\label{prop:bulk}
Let $d\geq 1$ and $\beta\in[0,\infty)$ and let $(A,\Phi)$ be distributed as a translation-invariant augmented $\beta$-arboreal gas Gibbs measure on $\Z^d$. The following hold:
    \begin{enumerate}
        \item The augmented connectivity relation $\xleftrightarrow{(A,\Phi)}$ has at most one infinite equivalence class a.s.
        \item If the augmented connectivity relation $\xleftrightarrow{(A,\Phi)}$ has an infinite equivalence class, then the subgraph of $\Z^d$ induced by this equivalence class is connected a.s.
    \end{enumerate}
\end{proposition}

It suffices to prove this in the case that the law of $(A,\Phi)$ is extremal in $\mathcal{M}_\beta^T$, taking a decomposition in terms of such extremal measures otherwise.

\medskip

 The proof of Proposition~\ref{prop:bulk} will make use of the following important fact, which follows from the work of Aldous and Lyons \cite{AL07} as explained in detail in \cite[Section 3]{MR3820434} and which is closely related to the classical work of Burton and Keane \cite{burton1989density}.

 \begin{proposition} \label{prop:ends}
   Let $d\geq 1$ and let $S$ be a translation-invariant random subgraph of $\Z^d$. Then every connected component of $S$ has at most two ends almost surely.
 \end{proposition}

 \medskip
 Fix $\beta\in(0,\infty)$, and $d\geq 2$ and let $\mathbb{Q}$  denote an extremal $\beta$-arboreal gas augmented Gibbs measure on $\Z^d$, and let $(A,\Phi)\sim\mathbb{Q}$.
 The Gibbs property tells us that for any $H\in\mathcal{S}^f(G)$, we have that 
 \[(A,\Phi)\sim (A,\Phi)\cup F \setminus E[H],\]
 $F$ has conditional law $\pr_{H,\beta}^{\Phi(H)}$ given $(A,\Phi)$. Since $\beta\in (0,\infty)$, this implies in particular that, conditional on $A\setminus E[H]$ and $\Phi(H)$, there is a.s.\ a positive probability that $A\cap E[H]=F^\prime$
 for any forest $F^\prime\in\mathcal{F}(H,\Phi(H))$. 
 This leads in particular to the following lemma.

 \begin{lemma}\label{lem:finite_energy}
   Fix $d\geq 2$, $\beta\in(0,\infty)$, let $\mathbb{Q}_\beta$ be an augmented $\beta$-arboreal gas Gibbs measure on $\Z^d$, and 
    % $G=(\Z^d,E_d)$.
     let $(A,\Phi)$ be distributed as $\mathbb{Q}$. 
          \begin{enumerate}\item If $H$ is a finite subgraph of $\Z^d$ then
% \begin{enumerate}\item (Deletion tolerance).
% \[
\begin{equation}
\label{eq:deletion_tolerance}
\mathbb{Q}_\beta(H \cap A = \emptyset \mid \mathcal{G}_H) >0 \qquad \text{\emph{a.s.}}
\end{equation}
 % (Merge tolerance)
     % the conditional probability 
     \item If $H$ is a finite \emph{connected} subgraph of $\Z^d$ then
     \begin{equation}
     \label{eq:merge_tolerance}
\mathbb{Q}_\beta(\text{\emph{all vertices of $H$ belong to the same augmented connectivity class}}\mid\mathcal{G}_H)>0 \qquad \text{\emph{a.s.}}
     \end{equation}
\end{enumerate}
     % Then the augmented connectivity relation $\stackrel{(A,\Phi)}{\longleftrightarrow}$ has at most two infinite equivalence classes almost surely.
 \end{lemma}

We refer to the property \eqref{eq:deletion_tolerance} of $\mathbb{Q}_\beta$ as \textbf{deletion tolerance} and the property \eqref{eq:merge_tolerance} as \textbf{merge tolerance}.

 \begin{proof}[Proof of Lemma~\ref{lem:finite_energy}]
 The deletion tolerance property \eqref{eq:deletion_tolerance} is an immediate consequence of the augmented Gibbs property since $\beta<\infty$. We now turn to the merge tolerance property \eqref{eq:merge_tolerance}.
Since $H$ is connected, $H/\Phi(H)$ is connected and therefore admits at least one spanning tree, which is given positive mass by the conditional measure $\mathbb{P}^{\Phi(H)}_{H,\beta}$ since $\beta>0$. On the event that the restriction of the arboreal gas to $H$ is equal to such a spanning tree, all vertices of $H$ belong to the same augmented connectivity class.
 \end{proof}
 
 The proofs in the remainder of this section and in the next will generally proceed by assuming that $(A,\Phi)$ satisfies a certain property with positive probability and then attempting to derive a contradiction.
 We will use the above observation to make local edits to $(A,\Phi)$, stitching together or separating infinite subgraphs as appropriate. Either ergodicity of $\pi_* \mathbb{Q}$, Proposition \ref{prop:ends}, or a combination thereof will then be used to generate the desired contradictions.

 \begin{remark}Several of the proofs in this section are of a similar flavour to those of \cite{hutchcroft2017indistinguishability,hutchcroft2016wired,timar2018indistinguishability}, which studied uniform spanning forests using a property known as \emph{update tolerance} or \emph{weak insertion tolerance}. There are however several important differences: 1) We need to understand the structure of the augmented connectivity relation, which was not a feature of those works. 2) Since $\beta<\infty$, we can use deletion tolerance to simplify several steps. 3) Our augmented Gibbs framework allows us to put many of the \emph{ad hoc} seeming parts of those papers on a more robust conceptual footing.
\end{remark}

We now begin the proof of Proposition~\ref{prop:bulk} in earnest. We begin by proving that $\stackrel{(A,\Phi)}{\longleftrightarrow}$ has at most two infinite  equivalence classes almost surely.

 \begin{lemma}\label{lem:unique_inf_class}
   Fix $d\geq 2$, $\beta\in[0,\infty)$, let $\mathbb{Q}_\beta$ be an extremal translation-invariant augmented $\beta$-arboreal gas Gibbs measure on $\Z^d$, and
    % $G=(\Z^d,E_d)$.
     let $(A,\Phi)$ be distributed as $\mathbb{Q}$. Then the augmented connectivity relation $\stackrel{(A,\Phi)}{\longleftrightarrow}$ has at most two infinite equivalence classes almost surely.
 \end{lemma}

\begin{figure}
\centering
\includegraphics[width=\textwidth]{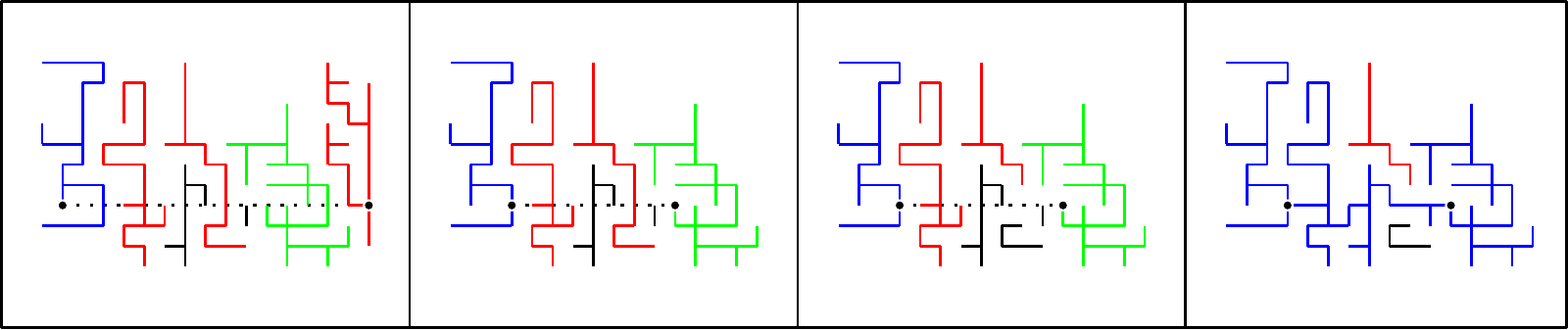}
\caption{Schematic illustration of the proof of Lemma~\ref{lem:unique_inf_class}. Infinite augmented connectivity classes are represented by colours, finite classes are black. Far left: a path $\gamma$ (dotted line) intersecting three distinct infinite augmented connectivity classes. Centre left: By shortening $\gamma$ if necessary, we may assume that $\gamma$ intersects exactly three distinct infinite augmented connectivity classes, two of which it intersects only at its endpoints. Centre right: By deleting finitely many edges from the configuration if necessary, we can make it so that each infinite augmented connectivity class intersecting $\gamma$ contains exactly one $A$-component intersecting $\gamma$. Far right: Using Lemma~\ref{lem:finite_energy}, we may glue together the components intersecting $\gamma$ to create a component with three or more ends, contradicting Proposition~\ref{prop:ends}.}
\label{fig:push_and_glue}
\end{figure}

 \begin{proof}[Proof of Lemma~\ref{lem:unique_inf_class}]
 An illustration of the proof is given in Figure~\ref{fig:push_and_glue}.
 The claim is trivial for $\beta=0$, so we restrict to the case $\beta>0$.
   Suppose for contradiction that the event \[E_1=\Bigl\{\text{$\stackrel{(A,\Phi)}{\longleftrightarrow}$ has three or more infinite equivalence classes}\Bigr\}\] 
   has positive probability. 
   For each $x\in \Z^d$, write $[x]$ for the equivalence class of $x$ under the augmented connectivity relation.
   Because $\mathbb{Q}_\beta(E_1)>0$, 
   there must exist three vertices $x,y,$ and $z$ such that $[x]$, $[y]$, and $[z]$ are all distinct with positive probability. 
   Fix three such vertices $x,y,z\in \Z^d$ and let $E_2$ be the event that this occurs.
    Since $\Z^d$ is $2$-connected, there exists a 
   simple path $\gamma$ in $\Z^d$ passing through $x$, $y$, and $z$. In particular, there must exist a finite simple path $\gamma$ that intersects at least three distinct infinite equivalence classes of the augmented connectivity relation with positive probability. Reducing the length of $\gamma$ if necessary, we may assume that, with positive probability, $\gamma$ intersects at least three infinite equivalence classes of the augmented connectivity relation, two of which it intersects only at its endpoints. Denote this event by $E_2$. Using deletion tolerance, it follows that, with positive probability, $\gamma$ intersects exactly three infinite clusters of $A$, all of which belong to distinct augmented equivalence classes, and with two of these clusters intersecting $\gamma$ only at its endpoints. Indeed, denoting this event by $E_3$, we note that if $E_2$ occurs but $E_3$ does not, so that the infinite augmented connectivity class $\mathscr{C}$ intersecting the interior of $\gamma$ contains multiple infinite $A$-components intersecting $\gamma$, then we can modify the configuration to make $E_3$ occur by choosing one of the infinite $A$ components that belongs to $\mathscr{C}$ and intersects $\gamma$, and deleting from $A$ all edges that are incident to $\gamma$ and belong to an $A$-component that belongs to $\mathscr{C}$ but is not equal to the one component we chose to keep.
Using merge tolerance allows us to glue together these three infinite $A$-components into a single infinite cluster by modifying $A$ on $\gamma$ in a way that preserves absolute continuity, and doing so creates a three ended component. 
Thus, there is a positive probability that $A$ contains a tree with at least three ends. Since $A$ is translation-invariant this contradicts Proposition \ref{prop:ends}, and so $\stackrel{(A,\Phi)}{\longleftrightarrow}$ has at most two infinite equivalence classes almost surely.
 \qedhere

% \medskip

   \end{proof}

The next step of the proof of Proposition~\ref{prop:bulk} is to prove Theorem~\ref{thm:ends}, which states that every infinite component of any translation-invariant $\beta$-arboreal gas Gibbs measure is one-ended almost surely for every $d\geq 1$ and $\beta\in (0,\infty)$.

\begin{proof}[Proof of Theorem \ref{thm:ends}] The claim is trivial if $\beta=0$ or $d=1$ so we may assume that $\beta>0$ and $d\geq 2$.
It suffices to prove the claim for measures of the form $\mathbb{P}=\pi_*\mathbb{Q}$ where $\mathbb{Q}=\mathbb{Q}_\beta$ is an extremal translation-invariant $\beta$-arboreal gas Gibbs measure on $\Z^d$. Let $(A,\Phi)\sim \mathbb{Q}$. By Proposition \ref{prop:ends}, all trees in $A$ have at most two ends almost surely, so we need only rule out the existence of two-ended trees.
 Note that if $e=\{x,y\}$ is an edge of $\Z^d$ we have by the augmented Gibbs property that
 \[\mathbb{Q}(e\in A \mid \mathcal{G}_e)=\frac{\beta}{1+\beta} \mathbbm{1}\left(\Phi(e)(x,y)=0\right),
 \]
 where we abuse notation to identify $e$ with the subgraph of $\Z^d$ having $\{x,y\}$ as its only vertices and $e$ as its only edge. Thus, we must have that $\Phi(e)(x,y)=0$ almost surely for every edge $e=\{x,y\}\in A$. It follows that, almost surely, if $A$ contains a two-ended tree $T$ and $e$ is an edge of $T$ such that $T\setminus e$ has two infinite connected components, then $(A,\Phi)\setminus e$ has one more infinite augmented connectivity class than $(A,\Phi)$ (where we allow both augmented subgraphs to have infinitely many infinite augmented connectivity classes in this statement). Thus, it follows by deletion tolerance that if $A$ has at least $n$ two-ended components with positive probability then $(A,\Phi)$ has at least $n+1$ infinite augmented equivalence classes with positive probability. Together with Lemma~\ref{lem:unique_inf_class}, this implies that $A$ has at most one two-ended component almost surely. On the other hand, if $A$ has exactly one two-ended component with positive probability then we have by deletion tolerance that $A$ has no two-ended components with positive probability. Since $\mathbb{Q}$ is extremal the law of $A$ is ergodic by Corollary~\ref{cor:A_ergodic}, and since the event that $A$ does not have any two-ended components is translation-invariant it must have probability $1$. \qedhere

\end{proof}

We next deduce that there is at most one infinite augmented connectivity class almost surely.
   
    \begin{lemma}\label{lem:unique_inf_class2}
   Fix $d\geq 2$, $\beta\in(0,\infty)$, let $\mathbb{Q}_\beta$ be an extremal translation-invariant augmented $\beta$-arboreal gas Gibbs measure on $\Z^d$, and
    % $G=(\Z^d,E_d)$.
     let $(A,\Phi)$ be distributed as $\mathbb{Q}_\beta$. Then the augmented connectivity relation $\stackrel{(A,\Phi)}{\longleftrightarrow}$ has at most one infinite equivalence class almost surely.
 \end{lemma}

 \begin{proof}[Proof of Lemma~\ref{lem:unique_inf_class2}]
   Suppose for contradiction that $(A,\Phi)$ has two infinite augmented connectivity classes with positive probability. Letting $H$ be a finite subgraph of $\Z^d$ that intersects both infinite equivalence classes with positive probability, we can use the merge tolerance of $(A,\Phi)$ to deduce that, with positive probability, $(A,\Phi)$ has a single infinite augmented equivalence class but $(A,\Phi)\setminus H$ does not. On this event there must exist an infinite component of $A$ with more than one end, contradicting Theorem~\ref{thm:ends}.
 \end{proof}

To complete the proof of Proposition \ref{prop:bulk}, we show that the induced subgraph $\mathrm{Tr}(I_\infty)$ is connected a.s.
 \begin{lemma}\label{lem:connected_trace}
   Fix $d\geq 2$, $\beta\in(0,\infty)$, let $\mathbb{Q}_\beta$ be an extremal augmented $\beta$-arboreal gas Gibbs measure on $\Z^d$, let $(A,\Phi)\sim \mathbb{Q}_\beta$ and let $I_\infty$ be the set of vertices of $\Z^d$ belonging to infinite clusters of $A$. If $I_\infty$ is non-empty then the induced subgraph $\mathrm{Tr}(I_\infty)$ is connected almost surely.
 \end{lemma}
 \begin{proof}[Proof of Lemma~\ref{lem:connected_trace}]
   The proof is similar to that of Lemma \ref{lem:unique_inf_class2}, but instead of attempting to connect infinite trees, we need (and, given Lemma \ref{lem:unique_inf_class2}, can) only connect their traces.
   Suppose for contradiction that the event \[E_1=\{\text{$\mathrm{Tr}(I_\infty)$ has three or more connected components}\},\] 
   has positive probability. 
   We will connect up the traces of three infinite trees from different components of $\mathrm{Tr}$ to give a component with at least three ends.
   Because $\mathbb{Q}(E_1)>0$, there exists a finite subgraph $H$ of $\Z^d$ that intersects at least three distinct infinite clusters of $\operatorname{Tr}(I_\infty)$ with positive probability. Using merge tolerance to force all elements of $H$ to belong to the same augmented connectivity cluster, it follows that, with positive probability, $\operatorname{Tr}(I_\infty(A))$ has a single component intersecting $H$ but $\operatorname{Tr}(I_\infty(A \setminus H))$ has at least three infinite components intersecting $H$. 
On this event we must have that $\operatorname{Tr}(I_\infty(A))$ contains a component with at least three ends.
   However $\mathrm{Tr}$ is connected and translation-invariant and so this contradicts Proposition \ref{prop:ends}, and so almost surely $\mathrm{Tr}$ has at most two infinite connected components almost surely.
 \end{proof}

 We are now ready to conclude the proofs of Proposition~\ref{prop:bulk} and Theorem~\ref{thm:resampling}.

\begin{proof}[Proof of Proposition \ref{prop:bulk}] It suffices to consider the case that $\beta>0$ and $d\geq 2$, the remaining cases being trivial. We may also assume that the law of $(A,\Phi)$ is extremal, taking an extremal decomposition otherwise. Once these reductions are made, the claims of Proposition \ref{prop:bulk} are exactly those of Lemmas \ref{lem:unique_inf_class} and \ref{lem:connected_trace}.
\end{proof}

\begin{proof}[Proof of Theorem \ref{thm:resampling}]
   Let $\pr$ be a translation-invariant $\beta$-arboreal gas Gibbs measure on $\Z^d$. Lemma~\ref{lemma:translation_tail} tell us that we can find a measure $\mathbb{Q}$ which is a translation-invariant augmentation thereof. Let $(A,\Phi)\sim\mathbb{Q}$. Propositions~\ref{prop:local_resampling_general} and \ref{prop:bulk} together imply that $\mathrm{Tr}(I_\infty)$ is a.s.\ connected and that the conditional distribution of the restriction of $A$ to $\mathrm{Tr}(I_\infty)$ given $I_\infty$ and the restriction of $A$ to $\mathrm{Tr}(I_\infty^c)$ is almost surely equal to some (possibly random) Gibbs measure for the uniform spanning tree on $\mathrm{Tr}(I_\infty)$. On the other hand, since $\mathrm{Tr}(I_\infty)$ is a translation-invariant random subgraph of $\Z^d$, it is a hyperfinite unimodular random rooted graph. As such, the results of Aldous and Lyons \cite[Proposition 8.14]{AL07} imply that its free and wired uniform spanning forests coincide, and hence that it has a unique Gibbs measure for the uniform spanning tree by Lemma~\ref{lemma:infty_free_wired}. This completes the proof.
\end{proof}

We end this section by observing the following corollary of Proposition~\ref{prop:bulk} and Theorem~\ref{thm:ends}.

\begin{corollary}\label{cor:augmentation_always_wired}
Let $d\geq 1$ and $\beta \in [0,\infty)$ and let $(A,\Phi)$ be distributed as a translation-invariant augmented $\beta$-arboreal gas Gibbs measure on $\Z^d$. Then the augmented subgraph $(A,\Phi)$ is almost surely equal to the wired augmentation of $A$ as defined in \eqref{eq:wired_augmentation}.
\end{corollary}
 
Corollary~\ref{cor:augmentation_always_wired} implies in particular that the completions of the sigma-algebras $\mathcal{I}$ and $\mathcal{I}_S$ coincide, which implies the following corollary in conjunction with Corollary~\ref{cor:A_ergodic}.

\begin{corollary}\label{cor:full_ergodicity}
 Every extremal translation-invariant augmented $\beta$-arboreal gas Gibbs measure on $\Z^d$ is ergodic for every $d\geq 1$ and $\beta \in [0,\infty]$.
\end{corollary}

% \subsection{All infinite trees are one-ended}
% \label{subsec:ends}

\section{Random walk intersections in unimodular random graphs}\label{section:RW}
In this section we prove Theorem \ref{thm:spanning_tree}, which states that  uniform spanning trees of unimodular random rooted subgraphs of $\Z^d$ are connected almost surely when  $d\leq 4$; by the results of Benjammini, Lyons, Peres and Schramm \cite{MR1825141,MR1997212} this is equivalent to the statement that two independent random walks on such a graph intersect infinitely often almost surely. This property is known as the \textbf{infinite intersection property}.
The proof is a combination of two results. First, in Section~\ref{subsec:intersection_criterion}, we establish, for general unimodular random rooted graphs whose degree has finite second moment, that two random walks intersect infinitely often almost surely if and only if their \emph{expected} number of intersections \emph{conditional on the rooted graph and one of the two walks} is infinite almost surely. Then, in Section~\ref{subsec:intersections_Zd}, we show that this condition is satisfied for random walks on unimodular subgraphs of $\Z^d$ for $d\leq 4$ using the theory of \emph{Markov-type inequalities}.

\medskip

Before getting started with the proof, we quickly review some relevant definitions and state a generalization of Theorem~\ref{thm:spanning_tree}.

\medskip

\noindent \textbf{Unimodular random rooted graphs.}
 A \textbf{rooted graph} is a pair $(G,\rho)$ where $G$ is a connected, locally finite graph and $\rho$ is a distinguished vertex of $G$ known as the root vertex; an isomorphism of graphs is an isomorphism of rooted graphs if it preserves the root. We define $\mathcal{G}_\bullet$ to be the space of isomorphism classes of rooted graphs, 
 which is equipped with the Borel sigma algebra induced by the 
  \emph{local topology} \cite{Curien,AL07}, in which two elements of $\mathcal{G}_\bullet$ are considered to be close if there exist large graph-distance balls around their root vertices which admit a graph isomorphism that preserves the root. The details of this construction are not important to us and can be found in e.g.\ \cite[Section 1.2]{Curien}. Similarly, we also have the space $\mathcal{G}_{\bullet\bullet}$ of (isomorphism classes of) doubly-rooted graphs $(G,\rho_1,\rho_2)$, with an ordered pair of distinguished root vertices $\rho_1,\rho_2\in V[G]$.
We say that a random variable $(G,\rho)$ taking values in $\mathcal{G}_\bullet$ is \textbf{unimodular} if it satisfies the \textbf{mass-transport principle}, meaning that
\[
\E{\sum_{v\in V[G]}F(G,\rho,v)}=\E{\sum_{v\in V[G]}F(G,v,\rho)}
\]
for every Borel measurable function $F:\mathcal{G}_{\bullet\bullet}\rightarrow[0,\infty)$.

Next we define the space of \textit{rooted subgraphs of $\Z^d$}; this definition is not standard.
 For any connected graph $G$ and $d\geq 1$, we say the function $\phi:V[G]\times V[G]\rightarrow \Z^d$ is an \textbf{embedding} of $G$ into $\Z^d$ if  $\phi(u,w)=\phi(u,v)+\phi(v,w)$ for every $u,v,w\in\Z^d$ (i.e.\ if $\phi$ is an additive cocyle), $\phi(u,w) = 0$ if and only if $u = w$, and $\norm{\phi(u,w)}_\infty = 1$ if $\{u,w\}\in E[G]$. A \textbf{rooted subgraph of $\Z^d$} is then a tuple $(G,\phi,\rho)$, where $G$, $\rho$ are as before, and $\phi$ is an embedding of $G$ into $\Z^d$. We denote the space of isomorphism classes of rooted subgraphs of $\Z^d$ by $\mathcal{S}_\bullet(\Z^d)$, which we endow with the Borel sigma algebra corresponding to the local topology, where for two elements to be close, the embeddings now also have to coincide in a large ball. Defining the space of doubly-rooted subgraphs $\mathcal{S}_{\bullet\bullet}(\Z^d)$ similarly, we say that a random tuple $(G,\phi,\rho)$ is unimodular if 
\[
\E{\sum_{v\in V[G]}F(G,\phi,\rho,v)}=\E{\sum_{v\in V[G]}F(G,\phi,v,\rho)}
\]
for every Borel measurable function $F:\mathcal{S}_{\bullet\bullet}(\Z^d)\rightarrow[0,\infty)$.

\begin{lemma}
If $\omega$ is a translation-invariant random subgraph of $\Z^d$, $K_0$ denotes the cluster of the origin in $\omega$, and we define a cocyle $\phi:V[K_0]\times V[K_0]\to \Z^d$ by $\phi(u,v)=u-v$, then $(K_0,\phi,0)$ is a unimodular random rooted subgraph of $\Z^d$.
\end{lemma}

\begin{proof}
The translation-invariance of the model implies that if $F:\mathcal{S}_{\bullet\bullet}(\Z^d)\rightarrow[0,\infty)$ is measurable then $F'(u,v)=\mathbb{E}[F(K_u,\phi,u,v)]$ satisfies $F'(u+x,v+x)=F'(u,v)$ for every $u,v,x\in \Z^d$, and the claim follows from the usual mass transport principle for $\Z^d$.
\end{proof}

Since unimodularity is preserved by conditioning on re-rooting invariant events, it follows that $(K_0,\phi,0)$ remains unimodular when we condition on it having size $n$ for any $n \in \N\cup\{\infty\}$ for which the relevant probability is positive. As such, Theorem~\ref{thm:spanning_tree} follows from the following more general theorem. (Examples of unimodular random rooted subgraphs of $\Z^d$ that do \emph{not} arise as a cluster in a translation-invariant model include the incipient infinite percolation cluster and the trace of a doubly-infinite random walk.)

 \begin{theorem} \label{thm:spanning_tree_general}
Let $d\leq 4$ and let $(G,\phi,\rho)$ be a unimodular random rooted subgraph of $\Z^d$. Then $G$ has the infinite intersection property almost surely. 
 \end{theorem}

 Equivalently, if $(G,\phi,\rho)$ is a unimodular random rooted subgraph of $\Z^d$ then the uniform spanning forest of $G$ is connected almost surely on the event that $G$ is infinite (the uniform spanning forest of $G$ being a.s.\ well-defined independently of boundary conditions by the results of \cite{AL07} as discussed in the proof of Theorem~\ref{thm:resampling}).

\subsection{A criterion for the infinite intersection property}
\label{subsec:intersection_criterion}

The goal of this subsection is to prove the following general proposition concerning intersections of random walks on general unimodular random rooted graphs.

 \begin{proposition} \label{prop:EToP}
 	Let $(G,o)$ be a unimodular random rooted graph which is almost surely connected and suppose that the second moment of the degree of the root is finite, i.e.\ $\mathbb{E}[\deg(o)^2]<\infty$. Let $X$ and $Y$ are two random walks on $G$, both started at $o$, that are conditionally independent given $(G,o)$. If
 	\begin{equation*}
    \mathbb{E}\left[\#\{i,j\geq0:X_i=Y_j\}\mid (G,o),Y \right]=\infty \qquad \text{ almost surely}
 	\end{equation*} 
   then $G$ has the infinite intersection property almost surely.
 \end{proposition}

The proof of this proposition is of a similar flavour to those of \cite{MR3399814,MR4364738}, which involve \emph{collisions} (where the two walks are at the same location \emph{at the same time}) rather than \emph{intersections} (where the two walks are at the same location but not necessarily at the same time).

\medskip

	% To prove this proposition, i
    We begin by establishing a lemma concerning random walks on \emph{deterministic} graphs.
    It will be convenient to work with two-sided rather than one-sided random walks. Given a connected, locally finite graph $G$ and two vertices $u,v\in V[G]$, we write $\mathbf{P}_{u,v}^{G}$ for the joint law of a pair of independent doubly-infinite random walks $(X_n)_{n\in \Z}$ and $(Y_n)_{n\in \Z}$ started at $u$ and $v$ respectively: Concretely, we let $X^+$, $X^-$, $Y^-$, and $Y^+$ be independent random walks on $G$, where $X^+$ and $X^-$ are started at $u$ and $Y^+$ and $Y^-$ are started at $v$, and define the two-sided random walks $(X_n)_{n\in \Z}$ and $(Y_n)_{n\in Z}$ by
	\[
	X_n = \begin{cases}
		X^+_n & n\geq 0\\
		X^-_{-n} & n\leq 0
	\end{cases} \qquad \text{ and } Y_n = \begin{cases}
        Y^+_n & n\geq 0\\
        Y^-_{-n} & n\leq 0.
    \end{cases}
	\]
    Given a subset $A$ of $\Z\times \Z$, we write $\operatorname{lex-max} A$ for the lexicographical maximum of $A$ when this maximum is well-defined.  
    The following lemma may be thought of as a time-reversal identity for the probabilities of these events.

\begin{lemma}
\label{lem:lex_max_identity}
Let $G=(V,E)$ be a transient, connected, locally finite graph, and let $o$ be a vertex of $G$. Then
    \begin{multline}
        \mathbf{P}_{o,o}^{G}\bigl(\operatorname{lex-max} \{(i,j):X_i=Y_j\}=(n,m)\bigr)
        \\=\sum_{v\in V}\frac{\deg(v)^2}{\deg(o)^2}\mathbf{P}_{v,v}^{G}(X_{-n}=Y_{-m}=o,\{X_i\}_{i\geq 0}\cap\{Y_j\}_{j>0}=\emptyset,\{X_i\}_{i> 0}\cap\{Y_j\}_{j\leq 0}=\emptyset)
    \end{multline}
    for every $n,m \geq 0$.
\end{lemma}

 (Here, the event ``$\operatorname{lex-max} \{(i,j):X_i=Y_j\}=(n,m)$'' implicitly includes the condition that the lexicographical maximum is well-defined.)

    \begin{proof}[Proof of Lemma~\ref{lem:lex_max_identity}]
Fix $n,m\geq 0$ and write
   \[
    B_{n,m}:=\big\{\operatorname{lex-max} \{(i,j):X_i=Y_j\}=(n,m)\big\}=\big\{X_n=Y_m,\{X_i\}_{i\geq n}\cap\{Y_j\}_{j>m}=\emptyset,\{X_i\}_{i> n}\cap\{Y_j\}_{j\leq m}=\emptyset\big\}.
    \]
 Decomposing according to the value of $X_n=Y_m$ yields that
	\begin{equation}\label{eq:sum_decomp}
		\mathbf{P}_{o,o}^{G}(B_{n,m})=\sum_{v\in V}\mathbf{P}_{o,o}^{G}(X_n=Y_m=v,\{X_i\}_{i\geq n}\cap\{Y_j\}_{j>m}=\emptyset,\{X_i\}_{i> n}\cap\{Y_j\}_{j\leq m}=\emptyset\big).
	\end{equation}
	Let $\mathbf{P}^G_o$ denote the marginal law of $(Y_n)_{n\in \Z}$ and abbreviate $\deg(v)=\mathrm{d}(v)$ for each vertex $v$ of $G$. For each $v\in V[G]$ and each doubly-infinite simple  $(x_n)_{n\in\Z}$ path in $G$ with
	$x_0=o$ and $x_n=v$
    we can compute that
		\begin{align*}
		&\mathbf{P}_{o}^G\big(Y_m=v,\, \{x_i\}_{i\geq n}\cap\{Y_j\}_{j>m}=\emptyset,\, \{x_i\}_{i> n}\cap\{Y_j\}_{j\leq m}=\emptyset\big)
        \\
		&=\mathbf{P}_{o}^G\big(\{x_i\}_{i> n}\cap \{Y_j\}_{j<0}=\emptyset \big)\mathbf{P}_{o}^G\big(\{x_i\}_{i> n}\cap\{Y_j\}_{0\leq j\leq  m}=\emptyset,\,Y_m=v\big)\mathbf{P}_{v}^G\big(\{x_i\}_{i\geq n}\cap\{Y_j\}_{j>0}=\emptyset\big)
    \\
		&=\mathbf{P}_{o}^G\big(\{x_i\}_{i> n}\cap \{Y_j\}_{j<0}=\emptyset \big)\left(\frac{\mathrm{d}(v)}{\mathrm{d}(o)}\mathbf{P}_{v}^G\big(\{x_i\}_{i> n}\cap\{Y_j\}_{0\leq  j\leq m}=\emptyset,\, Y_m=o\big) \!\right)\mathbf{P}_{v}^G\big(\{x_i\}_{i\geq n}\cap\{Y_j\}_{j>0}=\emptyset\big)
        \\
	&=\frac{\mathrm{d}(v)}{\mathrm{d}(o)}\mathbf{P}_{o}^G\big(\{x_i\}_{i> n}\cap \{Y_j\}_{j<0}=\emptyset \big) \mathbf{P}_{v}^G\big(\{x_i\}_{i> n}\cap\{Y_j\}_{-m\leq j\leq 0}=\emptyset,\, Y_{-m}=o\big) \mathbf{P}_{v}^G\big(\{x_i\}_{i\geq n}\cap\{Y_j\}_{j>0}=\emptyset\big)
  \\	
	&=\frac{\mathrm{d}(v)}{\mathrm{d}(o)} \mathbf{P}_{v}^G\big(\{x_i\}_{i> n}\cap\{Y_j\}_{j\leq 0}=\emptyset,\, Y_{-m}=o\big) \mathbf{P}_{v}^G\big(\{x_i\}_{i\geq n}\cap\{Y_j\}_{j>0}=\emptyset\big)\\	
	&=\frac{\mathrm{d}(v)}{\mathrm{d}(o)} \mathbf{P}_{v}^G\big(Y_{-m}=o,\, \{x_i\}_{i\geq n}\cap\{Y_j\}_{j>0}=\emptyset,\, \{x_i\}_{i> n}\cap\{Y_j\}_{j\leq 0}=\emptyset\big),
	\end{align*}
	where the first equality follows by independence of $\{Y_j\}_{j<0}$ and $\{Y_j\}_{j>0}$ and the Markov property of $\{Y_j\}_{j>0}$, the second equality follows by time-reversal for $\{Y_j\}_{j>0}$, the third equality follows as $\{Y_j\}_{j<0}$ and $\{Y_j\}_{j>0}$ are identically distributed, the penultimate inequality follows by the Markov property, and the final equality follows by independence of $\{Y_j\}_{j<0}$ and $\{Y_j\}_{j>0}$. 
Now, since $X$ and $Y$ are independent, letting $x=X$ gives
\begin{multline*}
\mathbf{P}_{o,o}^G(X_n=Y_m=v,\{X_i\}_{i\geq n}\cap\{Y_j\}_{j>m}=\emptyset,\{X_i\}_{i> n}\cap\{Y_j\}_{j\leq m}=\emptyset\big)\\ =\frac{\mathrm{d}(v)}{\mathrm{d}(o)}\mathbf{P}_{o,v}^G(X_n=v,Y_{-m}=o,\, \{X_i\}_{i\geq n}\cap\{Y_j\}_{j>0}=\emptyset,\,\{X_i\}_{i> n}\cap\{Y_j\}_{j\leq 0}=\emptyset\big),
\end{multline*}
 and applying a similar time-reversal to $X$ gives that
 \begin{multline*}
\mathbf{P}_{o,o}^G(X_n=Y_m=v,\{X_i\}_{i\geq n}\cap\{Y_j\}_{j>m}=\emptyset,\{X_i\}_{i> n}\cap\{Y_j\}_{j\leq m}=\emptyset\big)\\ =\frac{\mathrm{d}(v)^2}{\mathrm{d}(o)^2}\mathbf{P}_{v,v}^G(X_{-n}=Y_{-m}=o,\{X_i\}_{i\geq 0}\cap\{Y_j\}_{j>0}=\emptyset,\{X_i\}_{i> 0}\cap\{Y_j\}_{j\leq 0}=\emptyset).
\end{multline*}
The claim follows by substituting this into \eqref{eq:sum_decomp}.
    \end{proof}

    \begin{proof} [Proof of Proposition \ref{prop:EToP}]
    The claim holds trivially when $G$ is recurrent, so we may assume that $G$ is transient. Let $(X_n)_{n\in Z}$ and $(Y_n)_{n\in \Z}$ be doubly-infinite random walks started at $o$ that are conditionally independent given $(G,o)$. We assume that $\mathbb{E}\left[\#\{i,j\geq0:X_i=Y_j\}\mid (G,o),Y \right]=\infty$ almost surely and prove that in this case $\#\{i,j\geq0:X_i=Y_j\}=\infty$ almost surely.

    Recall that $B_{n,m}$ denotes the event that $\operatorname{lex-max} \{(n,m):X_n=Y_m\}=(n,m)$.
	Multiplying both sides of the identity of Lemma~\ref{lem:lex_max_identity} by $\deg(o)^2$, taking expectations and applying the mass-transport principle to the right-hand side gives
	\begin{align*}
		\mathbb{E}\big[\deg(o)^2\mathbbm{1}(B_{n,m})\big] &\geq \E{\sum_{v\in G}\mathbf{P}_{v,v}^G(X_{-n}=Y_{-m}=o,\{X_i\}_{i\geq 0}\cap\{Y_j\}_{j>0}=\emptyset,\{X_i\}_{i> 0}\cap\{Y_j\}_{j\leq 0}=\emptyset)}
        \\&= \E{\sum_{v\in G}\mathbf{P}_{o,o}^G(X_{-n}=Y_{-m}=v,\{X_i\}_{i\geq 0}\cap\{Y_j\}_{j>0}=\emptyset,\{X_i\}_{i> 0}\cap\{Y_j\}_{j\leq 0}=\emptyset)},
	\end{align*}
    where we bounded $\deg(v)\geq 1$ in the first line.
	Summing over $n,m \geq 0$ and using that the events $B_{n,m}$ are disjoint, we obtain that
	\begin{align*}
\E{\#\{i,j\leq0:X_i=Y_j\}\mathbbm{1}\big(\{X_i\}_{i\geq 0}\cap\{Y_j\}_{j>0}=\emptyset,\{X_i\}_{i> 0}\cap\{Y_j\}_{j\leq 0}=\emptyset\big)} \leq \mathbb{E}[\deg(o)^2] < \infty.
	\end{align*}
    Conditioning on the random rooted graph $(G,o)$ and the two-sided walk $Y$, conditional independence of $(X_i)_{i\leq 0}$ and $(X_i)_{i\geq 0}$ yields
	\begin{multline*}
        \mathbb{E}\Bigg[\E{\#\{i,j\leq0:X_i=Y_j\}\mid (G,o),Y}\cdot \pr\big(\{X_i\}_{i\geq 0}\cap\{Y_j\}_{j>0}=\emptyset,\{X_i\}_{i> 0}\cap\{Y_j\}_{j\leq 0}=\emptyset\mid(G,o),Y\big)\Bigg] < \infty
	\end{multline*}
    Since $\mathbb{E}\left[\#\{i,j\geq0:X_i=Y_j\}\mid (G,o),Y \right]=\infty$ almost surely by assumption, 
     the right hand side can only be finite if 
	\[
	\pr\big(\{X_i\}_{i\geq 0}\cap\{Y_j\}_{j>0}=\emptyset,\{X_i\}_{i> 0}\cap\{Y_j\}_{j\leq 0}=\emptyset\big)=0.
	\]
  Since the two events $\{\{X_i\}_{i\geq 0}\cap\{Y_j\}_{j>0}=\emptyset\}$ and $\{X_i\}_{i> 0}\cap\{Y_j\}_{j\leq 0}=\emptyset$ are conditionally independent given $(G,o)$ and $X$, and since $\pr\big(\{X_i\}_{i\geq 0}\cap\{Y_j\}_{j\leq 0}=\emptyset|(G,o),X)=\pr\big(\{X_i\}_{i\geq 0}\cap\{Y_j\}_{j\geq0}=\emptyset|(G,o),X)$, it follows that
\[
  \pr\big(\{X_i\}_{i\geq 0}\cap\{Y_j\}_{j>0}=\emptyset,\{X_i\}_{i> 0}\cap\{Y_j\}_{j\geq 0}=\emptyset\big)=0.
  \]
	In other words, two conditionally independent random walks $X$ and $Y$ started at $o\in G$ will almost surely satisfy $X_n=Y_m$ at some time $(n,m)$ with $n,m\geq 0$ and $(n,m)\neq (0,0)$. Since the random rooted graph $(G,o)$ is unimodular, the same statement holds almost surely for \emph{any} starting vertex $v\in G$ \cite[Proposition 11]{Curien}.
	 Now if $X$ and $Y$ are conditionally independent random walks on $G$ with arbitrary starting vertices, then the Markov properties of the random walks implies that for any $(n,m)\in\Z_{\geq 0}\times\Z_{\geq 0}$, the processes $(X_i)_{i\geq n}$ and $(Y_i)_{i\geq m}$ are jointly distributed as two conditionally independent random walks on $G$ started at $X_n$ and $Y_m$ respectively. In particular, together with our conclusion above, this implies that for any $n,m\geq 0$, the event 
   \[\{X_n\neq Y_m\}\cup\{\exists i\geq n,\, j\geq m: (i,j)\neq (n,m) \text{ and } X_i= Y_j\}\]
    occurs almost surely. If we now suppose that $X$ and $Y$ start at the same vertex, then we can use this fact inductively to construct two non-decreasing sequences of times $(T_i)_{i\geq 0}$ and $(S_i)_{i\geq 0}$ such that $(S_i+T_i)_{i\geq 0}$ is strictly increasing and $X_{T_i}=Y_{S_i}$ almost surely for every $i\geq 0$. Thus the proposition is proved.
\end{proof}

\begin{remark}
	This proposition certainly does not hold if the unimodularity assumption is removed. For instance, take two copies of $\Z^3$ attached by a single edge: The conditional expectation of the number of intersections is almost surely infinite, but the number of intersections has a positive probability of being finite due to the fact that the random walks may eventually remain in distinct copies of $\Z^3$. We are unsure if the analogous statement holds if we only require that $\mathbb{E}\left[\#\{i,j\geq0:X_i=Y_j\}\mid (G,o) \right]=\infty$ a.s.\ rather than $\mathbb{E}\left[\#\{i,j\geq0:X_i=Y_j\}\mid (G,o),Y \right]=\infty$ a.s.\
\end{remark}

\begin{remark}[Relaxing the second moment condition]
The proof of Proposition~\ref{prop:EToP} shows more generally that if $(G,o)$ is a unimodular random rooted graph with $\mathbb{E} [\deg(o)^\alpha]<\infty$ for some $0\leq \alpha \leq 2$ and $\mathbb{E} [\sum_{i,j=0}^\infty \mathbbm{1}(X_i=Y_j) \deg(Y_j)^{-2+\alpha}\mid (G,o),Y]=\infty$ almost surely then $G$ has the infinite intersection property almost surely.
\end{remark}

\subsection{Proof of Theorems~\ref{thm:spanning_tree} and \ref{thm:spanning_tree_general}}
\label{subsec:intersections_Zd}

In this section we complete the proof of Theorems~\ref{thm:spanning_tree} and \ref{thm:spanning_tree_general}, and hence also of Theorem~\ref{thm:uniqueness}, by proving the following proposition, which implies these theorems in conjunction with Proposition~\ref{prop:EToP} and Theorem~\ref{thm:resampling}.

\begin{proposition} \label{prop:dim34expinfinite}
	Let $1\leq d\leq 4$, let $(G,\phi,\rho)$ be a unimodular random random rooted subgraph of $\Z^d$, and let $X$ and $Y$ be two independent random walks on $G$ beginning at $\rho$. Then
	\begin{equation*} \label{eq:hyp} \mathbb{E}\left[\#\{i,j\geq0:X_i=Y_j\}\mid (G,\phi,\rho),Y \right]=\infty \text{ almost surely}.
	\end{equation*} 
\end{proposition}

The proof of this proposition will apply the theory of \emph{Markov-type inequalities}, which were first introduced by Ball \cite{MR1159828} in the context of the Lipschitz extension problem and have since been found to have many important applications to the study of random walk. We now give a quick review of the parts of the theory most relevant to us, referring the reader to \cite[Chapter 13.4]{LP:book} for further background.

\medskip

\noindent \textbf{Markov-type inequalities.}
A metric space $\mathcal{X}=(\mathcal{X},d)$ is said to have \textbf{Markov-type} $2$ with constant $C<\infty$ if 
% the following holds: 
for every finite set $S$, every irreducible reversible Markov chain $M$ on $S$, and every function $f:S\rightarrow \mathcal{X}$ the inequality
\[
\mathbb{E}\left[d\big(f(Y_0),f(Y_n)\big)^2\right]\leq C^2n \mathbb{E}\left[d\big(f(Y_0),f(Y_1)\big)^2\right]
\] 
holds for every $n\geq 0$,
where $(Y_i)_{i\geq 0}$ is a trajectory of the Markov chain $M$ with $Y_0$ distributed as the stationary measure of $M$. Similarly, a metric space $\mathcal{X}=(\mathcal{X},d)$ is said to have \textbf{maximal Markov-type} $2$ with constant $C<\infty$ if for every finite set $S$ and every irreducible reversible Markov chain $M$ on $S$, and every function $f:S\rightarrow \mathcal{X}$, we have that 
\begin{equation}
\label{eq:maximal_Markov_type}
\mathbb{E}\left[\max_{0\leq i\leq n}d\big(f(Y_0),f(Y_i)\big)^2\right]\leq C^2n \mathbb{E}\left[d\big(f(Y_0),f(Y_1)\big)^2\right]
\end{equation}
for each $n\geq 0$,
where, as before, $(Y_i)_{i\geq 0}$ is a trajectory of the Markov chain $M$ with $Y_0$ distributed as the stationary measure of $M$. Of particular importance to us will be the fact that $\R$ has maximal Markov-type 2 \cite[Theorem 13.15]{LP:book}, which implies by projecting onto each coordinate that $\R^d$ has maximal Markov-type 2 with the same constant for each $d\geq 1$, which implies the following inequality for unimodular random rooted subgraphs of $\Z^d$.

\begin{proposition} \label{prop:hyper_markov}
	Let $d\geq 1$ and let $(G,\phi,\rho)$ be a unimodular random rooted subgraph of $\Z^d$. If $Y$ is a random walk on $G$ started at $\rho$ then
	% we have 
	\[
	\mathbb{E}\left[\deg(\rho)\max_{0\leq i\leq n}\norm{\phi(Y_{ik},Y_0)}_\infty^2\right]\leq C^2 n \mathbb{E}\Bigl[\deg(\rho) \|\phi(Y_k,Y_0)\|\Bigr].
	\] 
	for each $n,k\geq 1$. Since $\|\phi(Y_k,Y_0)\| \leq 1$ and $1\leq \deg(\rho)\leq 2d$, it follows in particular that
	\[
	\mathbb{E}\left[\max_{0\leq i\leq n}\norm{\phi(Y_i,Y_0)}_\infty^2\right]\leq 2d C^2 n = C_0(d)^2n
	\] 
	for each $n\geq 1$, where $C_0(d)=C\sqrt{2d}$.
\end{proposition}

\begin{proof}[Proof of Proposition~\ref{prop:hyper_markov}]
This follows from the standard maximal Markov type inequality \eqref{eq:maximal_Markov_type} by using that unimodular random rooted subgraphs of $\Z^d$ are \emph{hyperfinite}. This means in particular that they can always be written as Benjamini-Schramm limits of \emph{finite} random rooted subgraphs of $\Z^d$, which are finite reversible Markov chains whose stationary measure is proportional to their degree. The details are very similar to the proof of \cite[Corollary 2.5]{MR4146545} and are omitted. 
\end{proof}

\begin{proof} [Proof of Proposition \ref{prop:dim34expinfinite}]
	Fix $\varepsilon\in(0,1)$ and let $C_0=C_0(d)$ be the constant from Proposition \ref{prop:hyper_markov}. Define constants $c_1=\sqrt{2C_0/\varepsilon}$ and $c_2=2/\varepsilon$, and define  sequences of times $t_n=4^n$, radii $r_n=\lceil c_1\cdot 2^n\rceil $, and Euclidean boxes $\Lambda_n=[-r_n,r_n]^d\subset\mathbb{R}^d$. Proposition \ref{prop:hyper_markov} and Markov's inequality give us that for each $n\geq 1$,
	\begin{equation} \label{eq:box_bound}
		\pr(\{\phi(\rho,Y_i)\}_{i\leq t_n}\subset \Lambda_n)\geq 1-\varepsilon.
	\end{equation}
	For each subset $A\subseteq\Z_{\geq 0}$ and $v\in\Z^d$, define the random variable $L_A(v)=\sum_{n\in A}\mathbbm{1}(Y_n=v)$ giving the number of times $i$ in $A$ such that $Y_i=v$, and define the partial Green's function $G_A(v)=\mathbb{E}^G[L_A(v)]$.
	We lower bound
	\begin{equation} \label{eq:lower_b}
		\mathbb{E}\left[\#\{i,j\geq0:X_i=Y_j\}\mid (G,o),Y \right]=\sum_{v\in G} G_{\Z_{\geq 0}}(v)L_{\Z_{\geq 0}}(v)\geq \sum_{n\geq 1}\sum_{v\in\Lambda_n} G_{[t_{n-1},t_n)}(v)L_{[t_{n-1},t_n)}(v),
	\end{equation}
where we write $v\in\Lambda_n$ as shorthand for $\phi(\rho,v)\in\Lambda_n$.
We aim to show that each sum over $\Lambda_n$ has good probability to contribute a constant to the total.
	To this end, for each $n\geq 1$ let $b_n=2^{-n(d-2)}/(4c_2(4c_1)^d)$ and let 
   \[U_n=\sum_{v\in\Lambda_n} L_{[t_{n-1},t_n)}(v)\mathbbm{1}( G_{[t_{n-1},t_n)}(v)< b_n).\] We can bound
	\begin{multline} 
			\sum_{v\in\Lambda_n} G_{[t_{n-1},t_n)}(v)L_{[t_{n-1},t_n)}(v) \geq \sum_{v\in\Lambda_n} G_{[t_{n-1},t_n)}(v)L_{[t_{n-1},t_n)}(v) \mathbbm{1}( G_{[t_{n-1},t_n)}(v)\geq b_n)\\
			\geq b_n \sum_{v\in\Lambda_n} L_{[t_{n-1},t_n)}(v) \mathbbm{1}( G_{[t_{n-1},t_n)}(v)\geq b_n)=b_n\Bigg[\sum_{v\in\Lambda_n} L_{[t_{n-1},t_n)}(v)-U_n\Bigg],\label{eq:split_decomp}
	\end{multline}
   and also have trivially that
	\begin{align*}
		\mathbf{E}^{G}[U_n]=\mathbf{E}^{G}\left[\sum_{v\in\Lambda_n} L_{[t_{n-1},t_n)}(v)\mathbbm{1}( G_{[t_{n-1},t_n)}(v)< b_n)\right]&=\sum_{v\in\Lambda_n}G_{[t_{n-1},t_n)}(v)\mathbbm{1}( G_{[t_{n-1},t_n)}(v)< b_n) \leq b_n\abs{\Lambda_n},
	\end{align*}
where we write $\abs{\Lambda_n}$ for the number of vertices $v\in G$ such that $\phi(\rho,v)\in \Lambda_n$. Since we also have that $\sum_{v\in\Lambda_n} L_{[t_{n-1},t_n)}(v) \geq t_n-t_{n-1}$ on the event that $\{\phi(\rho,Y_i)\}_{i\leq t_n}\subseteq \Lambda_n$, we have by \eqref{eq:box_bound} and  Markov's inequality  that
% so that
	% then by Markov's inequality,
	\begin{equation} \label{eq:U_bound}
		\pr\left(\sum_{v\in\Lambda_n} L_{[t_{n-1},t_n)}(v) \geq t_n-t_{n-1}\right) \geq 1- \eps \qquad \text{ and } \qquad \pr\bigl(U_n \leq c_2 b_n \abs{\Lambda_n}\bigr)\geq 1-\varepsilon
	\end{equation}
   for every $n\geq 1$. Since $c_2 b_n |\Lambda_n| \leq (t_n-t_{n-1})/2$ by choice of $b_n$, it follows from this and \eqref{eq:split_decomp} that
   \[
\mathbb{P}\left(\sum_{v\in\Lambda_n} G_{[t_{n-1},t_n)}(v)L_{[t_{n-1},t_n)}(v) \geq \frac{b_n (t_n-t_{n-1})}{2} \right) \geq 1-2\eps
   \]
   for every $n\geq 1$.
Now, we also have that $b_n (t_n-t_{n-1})$ is of order $2^{(4-d)n}$ and hence, since $d\leq 4$, that $\frac{b_n (t_n-t_{n-1})}{2}$ is bounded below by a positive constant $c_3=c_3(\eps)$. Fatou's lemma then implies that
\[
\mathbb{P}\left(\sum_{v\in G} G_{\Z_{\geq 0}}(v)L_{\Z_{\geq 0}}(v) = \infty\right) \geq \mathbb{P}\left(\limsup_{n\to\infty}\sum_{v\in\Lambda_n} G_{[t_{n-1},t_n)}(v)L_{[t_{n-1},t_n)}(v) \geq c_3(\eps)\right) \geq 1-2\eps
\]
for every $\eps>0$, and the claim follows since $\eps>0$ was arbitrary.
\end{proof}
\begin{proof}[Proof of Theorem \ref{thm:resampling}]
This is an immediate consequence of Propositions \ref{prop:EToP} and \ref{prop:dim34expinfinite}.
\end{proof}

\begin{proof}[Proof of Theorem \ref{thm:spanning_tree}]
This is an immediate consequence of Theorems \ref{thm:resampling} and \ref{thm:spanning_tree}.
\end{proof}
\begin{remark}
	For $d<4$ we can substitute the use of the Markov-type inequalities  in the proof of Proposition \ref{prop:dim34expinfinite} with the Varopoulos-Carne inequality, which implies that the maximal displacement bound $\max_{i\leq n}d(X_0,X_i)$ has order at most $\sqrt{n\log n}$ with high probability on any graph of at most polynomial volume growth.  As such,  Proposition \ref{prop:dim34expinfinite} and thus Theorem \ref{thm:spanning_tree} generalises easily to unimodular random rooted graphs whose balls have volume $O(n^d)$ for some $d<4$ (and with $\mathbb{E}[\deg(\rho)^2]<\infty$), without the need to have a unimodular embedding into $\Z^d$. The four-dimensional case is more delicate since this dimension is critical for $\Z^d$ to have the infinite intersection property, with each dyadic scale only contributing $O(1)$ intersections in expectation. We believe that it should be possible to extend Theorem~\ref{thm:spanning_tree} to unimodular random rooted graphs whose balls have volume $O(n^4)$ using the methods of
    Ganguly, Lee, and Peres \cite{MR3655957}, who proved that any unimodular random rooted graph of polynomial volume growth satisfies a diffusive estimate at infinitely many scales. To do this, one would need to improve their displacement estimate to a \emph{maximal} displacement estimate of the same order; we do not investigate this here.
\end{remark}

\addcontentsline{toc}{section}{References}

\setstretch{1}
\footnotesize{
	\bibliographystyle{abbrv}
	\bibliography{ArborealBib.bib}
}
	
\medskip

\noindent \sc{N.\ Halberstam: CCIMI, University of Cambridge,} \email{nh448@cam.ac.uk}\\
\noindent \sc{T.\ Hutchcroft: PMA, Caltech} \email{t.hutchcroft@caltech.edu}

\end{document}